\documentclass[final]{siamltex}
\usepackage{amsmath,amsfonts,amssymb}
\usepackage{latexsym}
\usepackage{epsfig,overpic}
\usepackage{subfigure}
\usepackage{color}
\usepackage{showlabels}

\newcommand{\bI}{\mathbf I}

\newcommand{\PP}{\mathcal P}

\newcommand{\cV}{\mathcal{V}}

\newcommand{\bbf}{\mathbf f}
\newcommand{\bg}{\mathbf g}

\newcommand{\bn}{\mathbf n}

\newcommand{\bu}{\mathbf u}
\newcommand{\bv}{\mathbf v}
\newcommand{\bw}{\mathbf w}

\newcommand{\bz}{\mathbf z}
\newcommand{\bT}{\mathbf T}
\newcommand{\bphi}{\boldsymbol{\phi}}

\newcommand{\M}{\mathcal M}

\newcommand{\T}{\mathcal T}
\newcommand{\V}{\mathcal V}
\newcommand{\Div}{\mathop{\rm div}}

\newcommand{\cD}{\mathcal D}
\newcommand{\cF}{\mathcal F}

\newcommand{\cO}{\mathcal O}

\newcommand{\cS}{\mathcal S}

\newcommand{\nn}{\mathbb{N}}

\renewcommand{\div}{\textrm{div}\ \!}

\DeclareGraphicsExtensions{.pdf,.eps,.ps,.eps.gz,.ps.gz,.eps.Y}

\newcommand{\la}{\left\langle}
\newcommand{\ra}{\right\rangle}

\newcommand{\deriv}[2]{\frac{\partial #1}{\partial #2}}
\newcommand{\ddt}[1]{\frac{\partial #1}{\partial t}}
\newcommand{\bsigma}{\boldsymbol{\sigma}}

\newtheorem{assumption}{Assumption}[section]

\newtheorem{remark}{Remark}[section]
\begin{document}
\title{A time dependent Stokes interface problem: well-posedness and space-time finite element discretization}
\author{Igor Voulis
\and Arnold Reusken\thanks{Institut f\"ur Geometrie und Praktische  Mathematik, RWTH-Aachen
University, D-52056 Aachen, Germany (reusken@igpm.rwth-aachen.de,voulis@igpm.rwth-aachen.de).}
}
\maketitle
\begin{abstract}  In this paper a time dependent Stokes problem that is motivated by  a standard sharp interface model for the fluid dynamics of  two-phase flows is studied. This  Stokes interface problem has discontinuous  density and viscosity coefficients and a pressure solution that is discontinuous across an evolving interface. This strongly simplified two-phase Stokes equation is considered to  be a good model problem for the development and analysis of finite element discretization methods for two-phase flow problems. In view of the \emph{un}fitted finite element methods that are often used for  two-phase flow simulations, we are particularly interested in a well-posed variational formulation of this Stokes interface problem in a Euclidean setting.
Such  well-posed weak formulations, which are not known in the literature, are the main results of this paper.  Different variants are considered, namely one with suitable spaces of divergence free functions, a discrete-in-time version of it, and variants in which the divergence free constraint in the solution space is treated by a pressure Lagrange multiplier. The discrete-in-time variational formulation involving the pressure variable for the divergence free constraint  is a natural starting point for a space-time finite element discretization. Such a  method is introduced and  results of numerical experiments with this method are presented.  
\end{abstract}
%\tableofcontents
%
\section[Introduction]{Introduction}\label{sec:introduction}
Let $\Omega \subset \Bbb{R}^d$ be an open bounded connected domain  and $I:=(0,T)$ a time interval.    On the space-time cylinder $\Omega \times I$ we consider the 
following  standard sharp interface model (in strong formulation) for the fluid dynamics of a  two-phase incompressible flow, cf.~\cite{HoFAbels2016,Pruss2016,Reusken}:
\begin{align}
 \left\{
  \begin{aligned}
    \rho_i(\deriv{\bu}{t}+(\bu\cdot\nabla)\bu) &=  \Div\bsigma_i + \bg_i  \\
    \Div\bu &= 0
  \end{aligned}
  \right. & \qquad\text{in $\Omega_i(t) $,\quad $i=1,2$},
   \label{eq:NSseparat} \\[2ex]
  [\bsigma\bn_\Gamma] = -\tau\kappa\bn_\Gamma  \quad \text{on}~~\Gamma(t), \label{coupl1} \\
 [\bu]= 0 \quad \text{on}~~\Gamma(t),
    \label{coupl2} \\
 V_\Gamma= \bu \cdot\bn_\Gamma \quad \text{on}~~\Gamma(t).
    \label{eq:IMcond}
\end{align}
Here $\Gamma(t)= \overline{\Omega_1(t)}\cap \overline{\Omega_2(t)}$ denotes the (sharp) interface,  $\bsigma_i= - p \bI +\mu_i  \big( \nabla \bu + (\nabla \bu)^T \big)$ the Newtonian stress tensor and $V_\Gamma$ is the normal velocity of the interface.  The density and viscosity, $\rho_i$ and $\mu_i, i=1,2,$ are assumed to be constant in each phase. The constant $\tau \geq 0$ is the surface tension coefficient and $\kappa$ is the mean curvature of $\Gamma$, i.e., $\kappa(x)=\Div \bn_\Gamma(x)$ for $x \in \Gamma$. Unknowns are the velocity $\bu= \bu(x,t)$, the pressure $p = p(x,t)$ and the (evolving) interface $\Gamma(t)$.  To make the problem well-posed one needs suitable initial and boundary conditions for $\bu$ and $\Gamma$. Due to the coupling of the interface dynamics and the fluid dynamics in the two bulk phases, this is a highly nonlinear problem. There is extensive literature on existence of solutions and well-posedness of different formulations of this problem. Most publications on these topics study 
quite regular solutions (in H\"older spaces) and deal with well-posedness locally in time or global existence of solutions close to equilibrium states (e.g., \cite{Solonnikov1995,Solonnikov2007,Solonnikov2012,Pruss2009,Pruss2011}). Often simplifying assumptions are used, for example, $\bg_i=0$, $\tau=0$ or constant density ($\rho_1=\rho_2$). In other studies weaker solution concepts are used, for example, in \cite{Nouri1995,Nouri1997} the notion of renormalized solutions \cite{DiPernaLions} of transport equations is used  to derive an existence result (for $\tau=0$) and  in \cite{Abels2007} existence of  so-called measure-valued varifold solutions is shown (for constant density). Here we do not give an overview of the extensive literature in this research field; for this we refer the interested reader to the literature discussion in the recent book  \cite{Pruss2016}.  

We are interested in the development and analysis of finite element  discretization methods for the two-phase flow problem given above. Finite element methods for this problem class can be found in, e.g.,  \cite{Reusken,Bansch2001,Croce2010,BotheReusken}. We are not aware of any literature in which rigorous error analysis of such finite element methods is presented. Only very few partial results, e.g. on discrete stability as in \cite{Bansch2001}, are known. This lack of analysis is clearly related to the strong nonlinearity of the problem \eqref{eq:NSseparat}-\eqref{eq:IMcond}. We also note that approaches and results available in the mathematical literature  on existence of solutions and well-posedness of this problem turn out not to be very useful for the analysis of finite element discretization methods.
In view of this, we introduce and analyze a much simpler (linear) Stokes interface problem which, however, is motivated by and closely related to the two-phase flow problem given above. We now derive this Stokes interface problem. In  almost all numerical simulation methods for  \eqref{eq:NSseparat}-\eqref{eq:IMcond} one uses an iterative decoupling technique in which the interface evolution is decoupled from the flow problems in the subdomains. For the interface representation and numerical propagation one can use, for example, the  level set method and given an approximation of $\Gamma(t)$ for $t$ in a (small) time interval one then discretizes the coupled Navier-Stokes equations in the subdomains. These Navier-Stokes equations are usually linearized by inserting a known approximation of the velocity in the first argument of the quadratic term $(\bu\cdot\nabla)\bu$. These two subproblems (interface propagation and solution of flow problem in the subdomains) can be coupled by several different iterative 
methods. This decoupling and linearization procedure motivates the following simplifying assumptions. 
%
%We wish to solve \eqref{eq:NSseparat}-\eqref{eq:IMcond} numerically. We are interested in a Eulerian space-time finite element discretization techniques for this type of problem. In a numerical method one can consider a linearization of \eqref{eq:NSseparat}-\eqref{eq:IMcond} and a fixpoint iteration method, e.g. \cite{Reusken,Bansch2001,Croce2010} \textbf{Needs winnowing} in order to find a solution. See \cite[Chapter 9]{Pruss2016} for a theoretical approach. In this paper we will discuss this linearization in a continuous setting. 
% For the introduction and analysis of these space-time methods we restrict to a significantly simpler setting. 
%The methods that we consider, however,  can also be applied to the problem formulated above. \\
%The following two simplifications are made in order to obtain a linear problem.
 Firstly, we assume a \emph{given}  sufficiently smooth (specified below) flow field  $\bw=\bw(x,t) \in \Bbb{R}^d$ on $Q:= \Omega \times I$, with $\Div \bw =0 $ on $Q$, which transports the interface. Instead of the interface dynamics condition $V_\Gamma= \bu \cdot\bn_\Gamma$ we impose $V_\Gamma= \bw \cdot\bn_\Gamma$. This implies that the interface evolution is completely determined by $\bw$ and the Navier-Stokes flow problem in the two subdomains $\Omega_i(t)$ is decoupled from the interface dynamics. Secondly, we use a linearization of the Navier-Stokes equation in which  $(\bu\cdot\nabla)\bu$ is replaced by $ (\bw\cdot\nabla)\bu$.
% In a numerical scheme, $\bw$ can be considered to be a (finite element) approximation of $\bu$, which is known from the previous step in the fixpoint iteration method.
Thus we obtain a time dependent (generalized) Stokes  problem (also called Oseen problem) in  each of the subdomains, with coupling conditions as in \eqref{coupl1}-\eqref{coupl2}. We introduce the usual notation for the material derivative along the flow field $\bw$: 
\[
\dot v := \frac{\partial v}{\partial t}+ \bw \cdot \nabla v.
\]
We also introduce the piecewise constant functions $\rho,\mu$ with $\rho(x,t):=\rho_i$, $\mu(x,t):=\mu_i$ in $\Omega_i(t)$ and the deformation tensor $D(\bu):=\nabla \bu + (\nabla \bu)^T$.
Thus we obtain the following much simpler linear problem: determine $\bu$ and $p$ such that
\begin{align}
 \left\{
  \begin{aligned} 
    \rho \dot \bu -  \Div(\mu D(\bu)) +\nabla p & =  \bg_i  \\
    \Div\bu &= 0 
  \end{aligned}
  \right. & \qquad\text{in $\Omega_i(t) $,\quad $i=1,2$},
   \label{NSA} \\[2ex]
  [(-p\bI + \mu D(\bu))\bn_\Gamma] = -\tau\kappa\bn_\Gamma  \quad \text{on}~~\Gamma(t), \label{coupl1A} \\
 [\bu]= 0 \quad \text{on}~~\Gamma(t),
    \label{coupl2A} 
\end{align}
combined with suitable initial and boundary conditions for $\bu$. We restrict to homogeneous Dirichlet boundary and initial  conditions for $\bu$:
\[
  \bu(x,t)= 0\quad \text{for}~~(x,t)\in \big(\partial\Omega \times I\big) \cup \big(\Omega \times \{0\}\big).
\]
% We allow a moving interface $\Gamma(t)$, but assume that the interface evolution is determined by the given velocity field $\bw$. 
%If $\bw$ is sufficiently  smooth,  the space-time interface
% $
% \mathcal{S}:=\bigcup_{t\in I}\Gamma(t)\times \{t\}
% $
%is a $C^0$-manifold. 
Both for the analysis and numerical simulations it is very convenient to reformulate this simplified model in a  \emph{one-fluid Stokes interface model} that combines the flow equations in the subdomains \eqref{NSA} and the interface conditions \eqref{coupl1A}-\eqref{coupl2A}. We consider this Stokes interface problem to be an interesting and relevant subproblem for the numerical simulation of the full two-phase flow problem \eqref{eq:NSseparat}-\eqref{eq:IMcond}. For example, a finite element method that is stable and accurate for this Stokes interface problem can be expected to be an efficient discretization for the Navier-Stokes flow equations (with small Reynold's numbers) in the full two-phase flow problem. 
  The main contribution of this paper  is the derivation of a well-posed space-time variational formulation of this Stokes interface model and, based on this, a (Galerkin) space-time finite element discretization. 

We mention a few relevant properties of the interface Stokes problem \eqref{NSA}-\eqref{coupl2A}. The discontinuity of the coefficients $\rho_i,\mu_i$ across the interface and the interface force induced by the surface tension in \eqref{coupl1A} lead, even if the data is otherwise smooth, to a discontinuity in the pressure $p$ and to a discontinuity in the derivative of the velocity $\bu$ on the space-time interface $\mathcal{S}$, cf. \cite{Pruss2016}. Hence, we have to deal with moving discontinuities. Typically the interface is not constant in time and  thus we do not have a tensor product structure. These properties make this interface Stokes problem significantly more difficult to solve numerically than a standard time-dependent Stokes equation. Even for this strongly simplified problem we are not aware of any rigorous (sharp) error bounds for finite element discretization methods.

As a first step towards such an error analysis we need a suitable well-posed variational formulation.  Concerning this we distinguish two different approaches. Firstly, the formulation and corresponding analysis is based on Lagrangian techniques, in which a suitable (coordinate) transformation is used to transform the given problem into one with a tensor product structure (i.e., a stationary interface).  Such an approach is used in e.g. \cite{Pruss2016} (various parabolic two-phase problems) or \cite{Saal} (free boundary Stokes problem).
Such Lagrangian formulations are useful in the context of ALE (arbitrary Lagrangian Eulerian) discretizations and fitted finite elements. For  a class of parabolic interface problems error bounds for fitted finite element methods have been derived in the literature, e.g. \cite{Zou1998}. Alternatively, one can consider a formulation and analysis in an Eulerian setting  (no coordinate transformations). Such formulations, which are standard for one-phase (Navier-)Stokes equations (\cite{ErnGuermond, Temam,Wloka}) are  better suited for unfitted finite element techniques. Finite element methods for fluid-stucture interaction have analysed in this setting in  \cite{Olshanskii2018,Takahashi2009}. If in the original two-phase flow problem an interface capturing method such as the very popular level set method is used, this very often leads to the application of unfitted finite element discretization methods for the flow problem (meaning that the triangulations are not fitted to the evolving interface).  This then requires special finite element spaces, for example an  XFEM \cite{Fries2010,Lehrenfeld2015}, unfitted FEM with a 
Nitsche penalty term \cite{Hansbo2002,Hansbo2009} or a CutFEM \cite{Burman2012,Burman2015,Hansbo2014}.  In this paper we restrict to the Eulerian approach.

Hence, for the time dependent Stokes interface problem described above we are interested in a well-posed variational formulation in an Euclidean setting, similar to those for one-phase (Navier-)Stokes equations known in the literature (\cite{ErnGuermond,Temam,Wloka}). For one-phase (Navier-)Stokes equations new well-posed space-time  formulations have been developed in the recent papers \cite{SchwabStevenson1,SchwabStevenson2}. These formulations do not cover the Stokes interface model described above, due to the lack of  a tensor product structure. It turns out that in particular the discontinuity in the mass density $\rho$ across the interface causes significant difficulties considering the analysis of well-posedness, as explained in Remark~\ref{Remdifficulty}. As a main contribution of this paper we develop an analysis resulting in a well-posed space-time variational formulation. Main results on  well-posedness are given in Corollary~\ref{corolmain}, Theorem~\ref{mainthm1} and Theorem~\ref{thmequivalence}.
 
Our analysis is rather different from the analyses used in the derivation of a well-posed variational  one-phase Stokes problems \cite{ErnGuermond, Temam,Wloka}. 

Based on this space-time variational formulation we propose an space-time unfitted finite element method. 
The method combines standard Discontinuous Galerkin time discretization \cite{Matthies,Thomee,Steinbach2018} with an XFEM or CutFEM approach  \cite{Fries2010,Lehrenfeld2015,Burman2015,Hansbo2014} to account for the jump in pressure across the space-time interface $\mathcal{S}.$   We present results of numerical experiments with this method. An error analysis of this method is a topic of current research and not considered in this paper.

The remainder of the paper is organized as follows. In Section~\ref{sectweak} we introduce a  variational formulation of the Stokes interface problem \eqref{NSA}-\eqref{coupl2A} in an obvious space-time Sobolev space of divergence free functions. In Remark~\ref{Remdifficulty} we explain why the analysis of well-posedness of this formulation is problematic. This motivates the introduction of other (related) spaces, for which  well-posedness of a variational formulation  can be proved. This analysis is presented in Section~\ref{sectmodvari}.  In Section~\ref{sectbroken} a standard discontinuous Galerkin approach is applied to derive a well-posed space-time variational formulation that allows a time stepping procedure. In Section~\ref{sectpressure} we study a space-time variational problem involving the pressure variable to satisfy the divergence free constraint. Based on this variational formulation we introduce an unfitted space-time finite element method in Section~\ref{sectUnfittedFEM} and 
give results of numerical experiments with this method. We finally give a summary and outlook in Section~\ref{sectoutlook}. 
 
\section{Space-time variational formulation} \label{sectweak}
We start with an assumption concerning the required smoothness of the space-time interface $\mathcal{S}:=\bigcup_{t\in I}\Gamma(t)\times \{t\}$ and the given velocity field $\bw$. 
\begin{assumption} 
 Throughout the paper we assume that $\mathcal{S}$ is a connected Lipschitz hypersurface in $\Bbb{R}^{d+1}$ and that the given velocity field $\bw$ is divergence free and $\bw \in C(\bar I;L^2(\Omega)^d)$.  The latter guarantees that the material derivative $\dot v = \frac{\partial v}{\partial t} +\bw \cdot \nabla v$ is well-defined  in a  weak sense as in \cite{DiPernaLions}. The piecewise constant density $\rho$ and the velocity field $\bw$ are assumed to satisfy the compatibility condition $\dot \rho=0$. 
 
  Furthermore, we make the assumption $\bw \in L^\infty(Q)^d$. This condition can be replaced by another (more natural) one which depends on the dimension $d$, cf. Remark~\ref{RemSobInbed}.
\end{assumption}
\ \\[1ex]
As  is usually done in the analysis of (Navier-)Stokes equations, we restrict to suitable  subspaces of divergence free velocity fields and thus eliminate the pressure. We derive well-posedness of a suitable variational formulation in these subspaces. Therefore we introduce the spaces
\begin{equation} \label{defX}
  \cV :=\{\, v \in H_0^1(\Omega)^d~|~\Div v=0\,\}, \quad X:= L^2(I;\cV).
\end{equation}
Assume that the strong formulation \eqref{NSA}-\eqref{coupl2A} has a sufficiently smooth solution $\bu$. Multiplication by test function $\bv \in X$ and partial integration then implies:
\begin{equation} \label{rt}
 (\rho \dot \bu, \bv)_{L^2} + (\mu D(\bu),D(\bv))_{L^2} = (\rho \bg, \bv)_{L^2}- \tau  \int_0^T \int_{\Gamma(t)}   \kappa \bn_{\Gamma} \cdot \bv \, ds \, dt , 
\end{equation}
where $(\cdot,\cdot)_{L^2}$ denotes the (vector) $L^2$ scalar product over the space-time cylinder $Q$. 
\begin{remark}\label{rmSurfaceForce} \rm  
Note that the second term in the right hand-side of \eqref{rt} corresponds to a force that acts only on the space-time interface $\mathcal{S}$. This induces a discontinuity in the pressure Lagrange multiplier. 
Below, instead of the specific right hand-side in \eqref{rt} we consider a generic $F \in X'$. If $\bg \in L^2(Q)^d$ and $\kappa\bn_\Gamma \in L^2(\mathcal{S})^d$ then the right hand-side satisfies $F \in X'$. In order to have the normal $\bn_\Gamma$ and the curvature $\kappa$ in the classical (strong) sense, we need additional ($C^2$) smoothness of $\Gamma(t)$. The regularity of $\Gamma(t)$ depends on the regularity of $\Gamma(0)$ and $\bw$ in the following way.
The advection field $\bw$ defines a Lagrangian flow $\Phi$ (see \cite{Crippa2009}): for a given $y \in \Omega$ the function $t \mapsto \Phi(y,t)$ is defined by the ODE system
\begin{equation}\label{flowPhi}
\begin{cases}
\ddt{  \Phi}(y,t) = \bw(\Phi(y,t),t), \quad t \in I,
\\
\Phi(y,0)=y.
\end{cases}
\end{equation}
For $\bw\in C^1(\bar I;C^2(\bar \Omega))^d$ this Lagrangian flow is uniquely defined and $\Phi \in C^2(\bar Q)^d$. This is known from classical Cauchy-Lipschitz theory, see \cite[\S 1.3]{Crippa2009}. Since $\Gamma(t) = \Phi(\Gamma(0),t)$, we can conclude that $\Gamma(t)$ is $C^2$ if $\bw\in C^1(\bar I;C^2(\bar \Omega)^d)$ and $\Gamma(0)$ is $C^2$. 

Weaker notions of curvature have been developed for cases with less smoothness. This issue, however, is not relevant for the well-posedness results in the remainder of the paper. 
%It merely expresses a sufficient condition to have a well-defined functional of the right-hand side of \eqref{rt}. 
\end{remark}
\ \\[1ex]

A suitable weak material derivative can be defined in the standard distributional sense. For this we first introduce further notation. 
Elements $\bu \in X$ have values $\bu(x,t):=\bu(t)(x) \in \Bbb{R}^d$, $(x,t)\in Q$. Due to the zero boundary values on $\partial \Omega$, the norm $\|v\|_1=(\|v\|_{L^2(\Omega)}^2 +\|\nabla v\|_{L^2(\Omega)}^2)^\frac12 $  on $H_0^1(\Omega)$ is equivalent to $|v|_1:=\|\nabla v\|_{L^2(\Omega)}$. In the remainder we use the latter norm, with corresponding scalar product denoted by $( \cdot,\cdot )_{1,\Omega}$ on $\cV$. The scalar product on $X$ is denoted by
\[
  ( \bu,\bv)_X:= \int_0^T \big( \bu(t),\bv(t)\big)_{1,\Omega} \, dt = \sum_{i=1}^d \int_0^T \int_\Omega \nabla u_i(t) \cdot \nabla v_i(t) \, dx \, dt  = ( \nabla \bu, \nabla \bv )_{L^2}.   
\]
Recall that $C_0^1(\Omega)^d \cap \cV$ is dense in $\cV$ and using the tensor product structure of  $X$ we get that
\begin{align} \label{defD0}
 \cD_{0}:= \{\, \sum_{i=1}^n g_i \bphi_i~|~ n \in \Bbb{N}, \, g_i \in C_0^\infty(I),\, \bphi_i \in C_0^1(\Omega)^d \cap \cV\, \} \subset C_0^1(Q)^d
\end{align}
is dense in $X$, i.e., $\overline{\cD_{0}}^{\|\cdot\|_X}=X$. 
For the case of an evolving  interface and with the material derivative in \eqref{NSA} it is natural to introduce the following weak \emph{material derivative} for functions from $X$.  
For $\bv \in X$ we define the functional $\rho \dot \bv$ by
\begin{equation} \label{r1dot}
   \la \rho \dot \bv, \bphi \ra  = - ( \rho \bv, \dot \bphi)_{L^2} \quad \text{for}~~\bphi \in \cD_0.
\end{equation}
Note that in the $L^2$ scalar product we use a weighting with the strictly positive piecewise constant function $\rho$.
We introduce the following analogon of the space $ \{\bv\in X~|~\ddt{\bv}\in X'\}$:
\[
  W=\{\, \bv \in X~|~ \rho \dot \bv \in X' \, \}, \quad \|\bv\|_W^2= \|\bv\|_X^2 +\|\rho\dot \bv\|_{X'}^2.
\]
An important difference between $ \{\bv\in X~|~\ddt{\bv}\in X'\}$ and $W$ is that, if $\rho$ varies with $t$ (i.e., $\rho_1 \neq \rho_2$ and the interface is not stationary),  the latter \emph{does not have a tensor product structure}.
\begin{remark} \rm
 Inserting the definition of the material derivative we get
\[
  (\rho \bv, \dot \bphi)_{L^2}= (\rho \bv, \frac{\partial \bphi}{\partial t})_{L^2} + (\rho \bv,\bw \cdot \nabla \bphi)_{L^2}.
\]
There is a constant $c$, which depends on $\|\bw\|_{L^{\infty}(Q)}$, such that $|(\rho \bv,\bw \cdot \nabla \bphi)_{L^2}| \leq c \|\bv\|_X \|\bphi\|_X $ for all $\bv \in X$, $\bphi \in \cD_0$. This implies  that  $\rho  \dot \bv \in X'$ iff  $\frac{\partial (\rho \bv)}{\partial t} \in X' $, and $\|\rho \dot \bv - \frac{\partial (\rho \bv)}{\partial t}\|_{X'}\leq c\|\bv\|_X$. Therefore the norms $\|\bv\|_{X}+\|\frac{\partial (\rho \bv)}{\partial t} \|_{X'}$ and $\|\bv\|_{W}$ are equivalent. 
\end{remark}
\ \\[1ex]
For smooth functions $\bv,\bphi \in C^1(\bar Q)^d \cap X$ we obtain, using \cite[Theorem II.6]{DiPernaLions} (applied to $\rho$ and $\bv\cdot \bphi$) and $\Div \bw = 0$, the following partial integration identity:
\begin{align}
 & \int_0^T \int_\Omega \rho \dot \bv\cdot \bphi + \rho \bv\cdot \dot \bphi = \int_0^T \int_\Omega \rho \dot{\overbrace{\bv\cdot \bphi}} \nonumber
\\ &=  \big(\rho(\cdot,T) \bv(\cdot,T), \bphi(\cdot,T)\big)_{L^2(\Omega)}
 -\big(\rho(\cdot,0) \bv(\cdot,0), \bphi(\cdot,0)\big)_{L^2(\Omega)} .\label{fundid}
\end{align}
%
%\begin{equation} \label{parti}
% \begin{split}
% & (\rho v, \dot \psi)_{L^2(Q)}  = \int_0^T \int_\Omega \rho v \dot \psi \, dx \, dt = \sum_{i=1}^2 \rho_i \int_0^T \int_{\Omega_i(t)} v \dot \psi \, dx \, dt \\
% & = \sum_{i=1}^2 \rho_i \int_0^T \int_{\Omega_i(t)}  \dot{ ( v \psi)} \, dx \, dt-  \int_0^T \int_{\Omega} \rho \dot v \psi \, dx \, dt  \\ & =\sum_{i=1}^2 \rho_i \int_0^T  
%\frac{d}{dt} \int_{\Omega_i(t)}    v \psi \, dx \, dt - (\rho \dot v, \psi)_{L^2(Q)} \\
% & = \big(\rho(\cdot,T) v(\cdot,T), \psi(\cdot,T)\big)_{L^2(\Omega)} -\big(\rho(\cdot,0) v(\cdot,0), \psi(\cdot,0)\big)_{L^2(\Omega)}- (\rho \dot v, \psi)_{L^2(Q)} . 
%\end{split}
%\end{equation}
% Thus, for smooth vector functions $ \bv,\bphi \in C^1(\bar Q)^d$ we get
%\begin{equation} \label{fundid}
%\begin{split}
% & (\rho \bv, \dot \bphi)_{L^2} + (\rho \dot \bv, \bphi)_{L^2} \\ & = \big(\rho(\cdot,T) \bv(\cdot,T), \bphi(\cdot,T)\big)_{L^2(\Omega)} -\big(\rho(\cdot,0) \bv(\cdot,0), \bphi(\cdot,0)\big)_{L^2(\Omega)}.
%\end{split}
%\end{equation}
For $\bphi \in \cD_0$ the boundary terms vanish, and thus we get
\begin{equation} \label{id2} \la \rho \dot \bv, \bphi \ra  = (\rho \dot \bv,  \bphi)_{L^2}  \quad \text{for}~~ \bv \in C^1(\bar Q)^d \cap X,\bphi \in \cD_0,
\end{equation}
which means that the weak material derivative  $\rho \dot \bv$ can be identified with the function $\rho \dot \bv$. By a continuity argument it follows that the result in \eqref{id2} also holds for all $\bphi \in X$.  

A natural weak formulation of \eqref{NSA}-\eqref{coupl2A} is as follows, cf. \eqref{rt}. Given $F \in X'$,  determine $\bu \in W$ with $\bu(0)=0$ and 
\begin{equation} \label{weak1}
\la \rho \dot \bu, \bv \ra + (\mu D(\bu),D(\bv))_{L^2} = F(\bv) \quad \text{for all}~~ \bv \in X. 
\end{equation}
\begin{remark} \label{Remdifficulty} \rm 
As noted above, the spaces $X$ and $W$ are very natural ones. We are, however,  not able to prove well-posedness of this formulation.  The key difficulty is to show that smooth functions are dense in $W$. For the case that the mass density $\rho$ is constant or  $\mathcal{S}$ does not depend on $t$ (stationary interface), density of smooth functions can be proved using mollification procedures in Bochner spaces as in e.g., \cite[Chapter 25]{Wloka}. For the general case, however, we do not have a tensor product structure and these techniques fail. We tried to develop a mollification technique in the full space-time cylinder $\Bbb{R}^{d+1}$. Such a mollification needs to satisfy a commutation property between mollification and distributional differentation \eqref{r1dot} (which involves the discontinuous function $\rho$) and furthermore must respect the divergence free property and the homogeneous Dirichlet boundary condition. We were not able to develop such a mollification technique.   If we would have a 
density of smooth functions property of $W$,  it can be shown  that there is a bounded  trace operator  $W \to L^2(\Omega)^d$, $\bu \to \bu(\cdot, t)$, which ensures that $\bu(0)$ is well-defined,  and partial integration rules can be derived. Well-posedness of \eqref{weak1} can then be derived using fairly standard arguments as in e.g. \cite[Chapter 26]{Wloka}. The density of smooth functions property, however, is an  open problem. 
\end{remark}
\ \\

\begin{remark} \rm 
In \eqref{weak1} we consider  a variational formulation in which a weak material derivative $\dot \bu$ is scaled with $\rho$, as in \eqref{NSA}. The scaling with $\rho$ (which does not have tensor product structure) causes significant difficultities in the theoretical analysis (Remark~\ref{Remdifficulty}). One might consider a rescaling of the momentum equation in \eqref{NSA} that eliminates the $\rho$ term in front of the material derivative $\dot \bu$. The two obvious possibilities are to introduce $\tilde p:= \rho^{-1} p$ or $\tilde \bu:= \rho \bu$. In both cases we rescale $\mu$, using $\tilde \mu:=\rho^{-1}\mu$. If we use $\tilde p$, then partial integration of the momentum equation (multiplied by a test function $\bv$) over the domain $\Omega=\Omega_1(t)\cup \Omega_2(t)$ results is an interface term of the form $\int_{\Gamma(t)}   [(-\tilde p\bI + \tilde \mu D(\bu))\bn_\Gamma] v\, ds$. This term can not be treated as a natural interface term, because in the interface condition \eqref{coupl1A} we have 
the quantities $p, \, \
\mu$ instead of $\tilde p,\, \tilde \mu$. If we use $\tilde \bu$, then  the term $\int_{\Gamma(t)}   [(- p\bI + \tilde \mu D(\tilde \bu))\bn_\Gamma] v\, ds$ occurs, which can be handled as a natural interface condition, due to $\tilde \mu D(\tilde \bu)= \mu D(\bu)$. However, from \eqref{coupl2A} we now obtain the interface condtion $[\rho^{-1} \tilde \bu]=0$, which implies that one has to use the space $H_0^1(\Omega_1(t)\cup \Omega_2(t))^d$ for $\tilde \bu$ instead of the (much nicer) space $H_0^1(\Omega)^d$ for $\bu$. Using these rescalings we are not able to derive a simpler analysis for well-posedness and therefore we keep the original formulation  \eqref{NSA}-\eqref{coupl2A}, which is closer to physics. 
\end{remark}

\section{Space-time variational formulation in modified spaces}  \label{sectmodvari}
 As explained in Remark~\ref{Remdifficulty}, we encounter difficulties in the analysis of well-posedness of the variational formulation \eqref{weak1} using the space $W$. In this section we study a variational formulation as in \eqref{weak1}, \emph{but with} $W$ \emph{replaced by a (possibly) smaller space} $V$ (introduced below).
The structure of the analysis is as follows. In Section~\ref{SectspacesUV} we introduce further spaces $U$ and $V$ and derive properties of these spaces. In Section~\ref{SectweakU} we study an intermediate weak formulation and derive a well-posedness result.  In Section~\ref{SectweakV} we introduce  and analyze the final desired weak formulation analogon of \eqref{weak1}, with $W$ replaced by $V$. The main results are given in Corollary~\ref{corolmain} and Theorem~\ref{mainthm1}.
\subsection{Spaces $U \subset V \subset W$} \label{SectspacesUV}
For $\bv \in X= L^2(I;\cV)$ its weak derivative $\ddt{\bv} \in \cD'(I;\cV')$ is defined in the usual distributional sense \cite{Wloka}. We define the spaces 
\begin{align}
  U&:= \{\, \bv \in X~|~ \ddt{\bv} \in L^2(I; L^2(\Omega)^d)\,\}, ~~\text{with norm}~~\|\bv\|_U^2 = \|\bv\|_X^2+\| \ddt{\bv}\|_{L^2}^2, \label{defU}\\
  V &:= \overline{U}^{\|\cdot\|_W},\quad \text{with norm}~~\|\bv\|_V=\|\bv\|_W.\label{defV}
\end{align}
These are Hilbert spaces with continuous embeddings 
\[ U \rightarrow V \rightarrow W.
\]
The norm $\|\cdot\|_U$ is equivalent to $\|\cdot\|_{H^1(Q)^d}$.
 The space $U$ has a tensor product structure and we can use standard arguments  to show that smooth functions are dense in $U$.  More precisely, let $\cD(\cV)$ be the space of all functions $\bv: \Bbb{R} \to \cV$ which are infinitely differentiable and have a compact support. Then, cf. Lemma 25.1 in \cite{Wloka},  $\cD(\cV)_{|I}$ is dense in $U$. Using the density of $C_0^1(\Omega)^d \cap \cV$ in $\cV$ we obtain that the space of smooth functions $\cD(C_0^1(\Omega)^d \cap \cV)_{|I}$ is dense in $U$. From  the density of $U$ in $V$ we thus get the density of smooth functions in $V$:
\begin{equation} \label{density}
  \overline{\cD(C_0^1(\Omega)^d \cap \cV)_{|I}}^{\|\cdot\|_W}= V.
\end{equation}
\begin{remark} \rm
It seems reasonable (based on analogous results for the tensor product case) to claim that $\overline{\cD(C_0^1(\Omega)^d \cap \cV)_{|I}}^{\|\cdot\|_W}=W$ holds, i.e., $V=W$. We are, however, not able to prove this claim, cf. Remark~\ref{Remdifficulty}.  Note that the well-posedness result derived  for $V$ in Corollary~\ref{corolmain} below implies that either $V \neq W$ or the well-posedness result holds for $W$.
\end{remark}
\ \\

Using the density result \eqref{density}
important properties of $V$ are derived in the following lemma.
\begin{lemma} \label{lem4}
\begin{description}
\item(i)~For a.e. $t \in [0,T]$ the trace operator $\bu \to \bu(\cdot,t)= \bu(t)$ can be extended to a bounded linear operator from $V$ into $L^2(\Omega)^d$.  Moreover, the inequality
\begin{equation} \label{rrt}
 \sup_{0 \leq t\leq T} \|\bu(t)\|_{L^2(\Omega)} \leq c \|\bu\|_V \quad \text{for all}~~ \bu \in V,
\end{equation}
holds with a constant $c$ independent of $\bu$.
\item(ii) For all $\bu,\bv \in V$, the following integration by parts  identity holds:
\begin{equation} \label{partint}
%\begin{split}
   \la \rho \dot \bu,\bv\ra +\la \rho \dot \bv, \bu\ra   =  \big(\rho(T) \bv(T), \bu(T)\big)_{L^2(\Omega)} -\big(\rho(0) \bv(0), \bu(0)\big)_{L^2(\Omega)}.
%\end{split}
\end{equation}
\end{description}
\end{lemma}
\begin{proof}  Take $t \in [0,\frac12 T]$ (the case $t\in [\frac12,T]$ can be treated with very similar arguments). Define $t_e:=t+\frac14 T$, $\tilde I:=(t,t_e)$, $\tilde Q:= \Omega \times  \tilde I \subset Q$. 
%The space $X$  can be defined with respect to the smaller interval $\tilde I$. The corresponding space is denoted by $\tilde X, \tilde Y$. One easily checks that if $u \in Y$, then $u \in \tilde Y$ and $\|u\|_{\tilde Y} \leq \|u\|_{Y}$ holds.  
It suffices to prove the result in \eqref{rrt} for the dense subspace $\cD:=\cD(C_0^1(\Omega)^d \cap \cV)_{|I}$ of smooth functions. Take $\bu \in \cD$. The partial integration identity \eqref{fundid} on $\tilde Q$ yields
\begin{equation} \label{ppp}
 \|\rho(t_e)^\frac12 \bu(t_e)\|_{L^2(\Omega)}^2 -\|\rho(t)^\frac12 \bu(t)\|_{L^2(\Omega)}^2 = 2 (\rho \dot \bu,\bu)_{L^2(\tilde Q)}. 
\end{equation}
Let $\sigma$ be a smooth decreasing scalar function with compact support and $\sigma(t)=1$, $\sigma(t_e)=0$. Note that $\sigma \bu \in \cD$ holds. If in \eqref{ppp} we use $\sigma \bu$, instead of $\bu$, we get, with $\rho_{\min} :=\min\{\rho_1,\rho_2\}$:
\begin{align*}
 \|\bu(t)\|_{L^2(\Omega)}^2 & \leq \rho_{\min}^{-1}\|\rho(t)^\frac12 \bu(t)\|_{L^2(\Omega)}^2 = 2 \rho_{\min}^{-1}|(\rho \dot{(\sigma \bu)},\sigma \bu)_{L^2(\tilde Q)}| \\
 & \leq 2 \rho_{\min}^{-1} \big( |(\rho \sigma' \sigma \bu,\bu)_{L^2(\tilde Q)}| +|(\rho \dot \bu, \sigma^2 \bu)_{L^2(\tilde Q)}| \big).
\end{align*}
Note that $|(\rho \sigma' \sigma \bu,\bu)_{L^2(\tilde Q)}| \leq c \|\bu\|_{L^2(\tilde Q)}^2 \leq c \|\bu\|_X^2 \leq c \|\bu\|_V^2 $ holds. Furthermore, with $\tilde X:= L^2(\tilde I;\cV)$, and extending $\bv \in  \tilde X$ by zero outside $\tilde I$, we have:
\begin{align*} |(\rho \dot \bu, \sigma^2 \bu)_{L^2(\tilde Q)}| & \leq \sup_{\bv \in \tilde X}\frac{(\rho \dot \bu, \bv)_{L^2(\tilde Q)}}{\|\bv\|_{\tilde X}} \|\sigma^2 \bu\|_{\tilde X} = \sup_{\bv \in \tilde X}\frac{(\rho \dot \bu, \bv)_{L^2( Q)}}{\|\bv\|_{X}} \|\sigma^2 \bu\|_{\tilde X} \\ &  \leq  c  \sup_{\bv \in X}\frac{(\rho \dot \bu, \bv)_{L^2( Q)}}{\|\bv\|_X} \| \bu\|_X  \leq c \|\rho \dot \bu \|_{X'}\| \bu\|_X  \leq c \|\bu\|_{V}^2.
\end{align*}
 Thus we get
\[
 \|\bu(t)\|_{L^2(\Omega)} \leq c \|\bu\|_{V},
\]
with a constant (depending on $T$) that is independent of $\bu \in \cD$. Due to density of $\cD$ this proves the result in \eqref{rrt}, and thus (i).\\
We consider (ii). Due to density and the continuity result in \eqref{rrt} it suffices to prove \eqref{partint} for $\bu,\bv \in \cD$. The identity in \eqref{id2}  holds for $\bphi \in X$ and thus for   $\bu,\bv \in \cD$ it follows from \eqref{id2} that 
\[ \la\rho \dot \bu,\bv\ra +\la \rho \dot \bv,\bu\ra =(\rho \dot \bu,\bv)_{L^2} + (\rho \dot \bv, \bu)_{L^2}.
\]
From this and the partial integration identity \eqref{fundid}  the result \eqref{partint} follows. 
\end{proof}

\subsection{Well-posed  space-time variational formulation in $U$} \label{SectweakU}
We define $U_0:=\{\, \bu \in U~|~\bu(0)=0\,\}$, where $\bu(0)$ is well-defined (in $L^2(\Omega)$ sense) due to \eqref{rrt}. In the following theorem we treat a variational problem with a sufficiently smooth right hand-side $\bbf$ and a bilinear form $a(\cdot,\cdot)$ on $\cV \times \cV$ that is \emph{in}dependent of $t$. These assumptions are such that we can apply a standard Galerkin procedure to show existence of a unique solution in $U_0$. This intermediate problem will be used in the next section to derive well-posedness of a weak formution as in \eqref{weak1}, with $W$ replaced by the space $V \subset W$.
\begin{theorem} \label{thm1}
 Take $\bbf \in C(\overline{I};L^2(\Omega)^d)$ and let $a(\cdot,\cdot)$ be a continuous elliptic bilinear form on $\cV \times \cV$ (with norm $|\cdot|_1$) that does not depend on $t$. Then there exists a unique $\bu \in U_0$ such that
\begin{equation} \label{Res1}
 (\rho \dot \bu, \bv)_{L^2} + \int_0^T a(\bu(t),\bv(t)) \, dt = \int_0^T (\bbf(t),\bv(t))_{L^2(\Omega)} \, dt \quad \text{for all}~~\bv \in X.
\end{equation}
Furthermore
\begin{equation} \label{Res2}
  \|\bu\|_{U} \leq c \|\bbf\|_{L^2}
\end{equation}
holds, with a constant $c$ independent of $\bbf$.
\end{theorem}
\begin{proof} The proof is based on a standard Galerkin technique known in the literature, e.g. \cite{Evans}. 
 Let $(\bv_k)_{k \geq 1}$ be a total orthonormal set in $\cV$ and define $\cV_m:={\rm span}\{\bv_1, \ldots, \bv_m\}$, $X_m:=L^2(I; \cV_m)$. We consider the following problem: determine $\bu_m \in X_m$ with $\bu_m(0)=0$ and such that:
 \begin{equation} \label{pp}
  (\rho \dot \bu_m,\bv)_{L^2} + \int_0^T a(\bu_m(t), \bv(t)) \, dt = \int_0^T (\bbf(t),\bv(t))_{L^2(\Omega)} \, dt \quad \text{for all}~~\bv \in X_m.
 \end{equation}
Using the representation $\bu_m(t)= \sum_{j=1}^m g_j(t) \bv_j$ and with $\bg_m(t):=(g_1(t),\ldots,g_m(t))^T$ this problem can be reformulated as a system of ODEs:
\begin{equation} \label{ODE}
 \begin{split}
   & M_m(t) \ddt{\bg_m(t)} + B_m(t)\bg_m(t)= F_m(t) \\
   & \bg_m(0)=0, 
 \end{split}
\end{equation}
with a symmetric positive definite matrix $M_m \in C(\bar{I};\Bbb{R}^{m\times m})$, $(M_m(t))_{i,j}=(\rho(t) \bv_j,\bv_i)_{L^2(\Omega)}$ and $B_m \in C(\bar{I};\Bbb{R}^{m\times m})$, $ (B_m(t))_{i,j}= (\bw(\cdot,t)\cdot \nabla \bv_j,\bv_i)_{L^2(\Omega)} + a(\bv_j,\bv_i)$ and $F_m\in C(\bar{I};\Bbb{R}^{m})$, $ (F_m(t))_{i}= (\bbf(t),\bv_i)_{L^2(\Omega)}$. Standard theory for ODEs implies that \eqref{ODE} has a unique solution $\bg_m \in C^1(\bar{I})^m$, and thus \eqref{pp} has a unique solution $\bu_m$. We take $\bv=\bu_m$ in \eqref{pp}:
\[
 (\rho \dot \bu_m, \bu_m)_{L^2} + \int_0^T a(\bu_m(t),\bu_m(t)) \, dt = (\bbf,\bu_m)_{L^2} .
\]
The ellipticity of $a(\cdot,\cdot)$ on $\cV$ implies that $\int_0^T a(\bu_m(t),\bu_m(t)) \, dt \geq \gamma \|\bu_m\|_X^2$ for a $\gamma >0$ independent of $\bu_m$. Combining this with partial integration, a Cauchy inequality and $\bu_m(0)=0$ yields
\[
 \|\rho^\frac12(T) \bu_m(T)\|_{L^2(\Omega)}^2 + \gamma \|\bu_m\|_X^2 \leq \|\bbf\|_{L^2} \|\bu_m\|_X,
\]
which implies a uniform bound $\|\bu_m\|_X \leq \gamma^{-1} \|\bbf\|_{L^2}$. 

We take $\bv= \ddt{ \bu_m}= \sum_{j=1}^m g_j'(t) \bv_j \in X_m$ in \eqref{pp}, and thus get:
\[
 (\rho \ddt{\bu_m},\ddt{\bu_m})_{L^2} + \int_0^T a(\bu_m(t),\ddt{\bu_m}(t) )\, dt = (\bbf, \ddt{\bu_m})_{L^2}  - (\rho \bw \cdot \nabla \bu_m,\ddt{d\bu_m})_{L^2}.  
\]
From $a(\bu_m(t),\bu_m(t))= \bg_m(t)^T A \bg_m(t)$, with $A_{i,j}=a(\bv_i,\bv_j)$ and $\bg_m(0)=0$ it follows that $ a(\bu_m(0),\bu_m(0))=0$. Using this we get
\begin{equation} \label{lp}
\int_0^T a(\bu_m(t),\ddt{\bu_m}(t) )\, dt = \frac12 \int_0^T \ddt{} a(\bu_m(t),\bu_m(t))\, dt = \frac12 a(\bu_m(T),\bu_m(T)) \geq 0.
\end{equation}
Using Cauchy inequalities and the uniform bound $\|\bu_m\|_X \leq \gamma^{-1} \|\bbf\|_{L^2}$ we obtain $\|\ddt{\bu_m}\|_{L^2} \leq c \|\bbf\|_{L^2}$ with a constant $c$ which only depends on $\rho, \|\bw \|_\infty$ and $\gamma$. Hence we have a uniform boundedness result 
\begin{equation} \label{np} \|\bu_m\|_U \leq c \|\bbf\|_{L^2}.
\end{equation}
Hence there is a subsequence, which we also denote by $(\bu_m)_{m \geq 0}$, that weakly converges $\bu_m \rightharpoonup \bu \in U$, which implies $\bu_m \rightharpoonup \bu$ in $X$ and $\ddt{\bu_m} \rightharpoonup \ddt{\bu}$ in $L^2(Q)$. Passing to the limit and using continuity arguments we conclude that $\bu \in U$ satisfies \eqref{Res1}. We now show that $\bu(0)=0$ holds, i.e., $\bu \in U_0$. Take an arbitrary $\bv \in C^1(\bar{I};\cV_N) \subset X_N$ with $\bv(T)=0$. From \eqref{Res1} and partial integration we obtain
\begin{equation} \label{o1}
 - (\rho \bu, \dot \bv)_{L^2} +   \int_0^T a(\bu(t),\bv(t)) \, dt = \int_0^T (\bbf(t),\bv(t))_{L^2(\Omega)} - (\rho(0)\bu(0),\bv(0))_{L^2(\Omega)}.
\end{equation}
We also get from \eqref{pp}, for $m \geq N$, and using $\bu_m(0)=0$:
\begin{equation} \label{o2}
 - (\rho \bu_m, \dot{\bv})_{L^2} +   \int_0^T a(\bu_m(t),\bv(t)) \, dt = \int_0^T (\bbf(t),\bv(t))_{L^2(\Omega)}\,dt.
\end{equation}
Comparing \eqref{o1}, \eqref{o2} and using $\bu_m \rightharpoonup \bu$ in $U$ it follows that $((\rho(0)\bu(0),\bv(0))_{L^2(\Omega)}=0$ holds. This implies $\bu(0)=0$ in $L^2(\Omega)$.
To show the uniqueness of $\bu$ we take $\bbf=0$ and $\bv=\bu$ in \eqref{Res1}: 
\[
 (\rho \dot \bu, \bu)_{L^2} + \int_0^T a(\bu(t), \bu(t))\, dt=0.
\]
Using $(\rho \dot \bu, \bu)_{L^2}= \frac12 \|\rho^\frac12(T) \bu(T)\|_{L^2(\Omega)}^2$ and the ellipticity of $a(\cdot,\cdot)$ it follows that $\|\bu\|_X=0$, hence we have uniqueness.
The bound in \eqref{Res2} follows from \eqref{np}.
\end{proof}
\ \\[1ex]
The assumption that the bilinear form  $a(\cdot,\cdot)$ is independent of $t$ is used for the derivation of a bound on $\|\ddt{\bu}\|_{L^2}$, for which the estimate in \eqref{lp} is a key ingredient. 
Using a standard Gronwall argument in \eqref{lp}, one can derive a similar result if the bilinear form $a(t;\cdot,\cdot)$ is time dependent and  differentiable with respect to time. Such arguments, however, fail when $a(t;\cdot,\cdot)$ is not differentiable,  which is the case we consider. The extension to a time-dependent bilinear form $a(t;\cdot,\cdot)$, with a possibly nonsmooth dependence on $t$, is treated in Theorem~\ref{mainprop1}.
\begin{remark}\label{RemSobInbed} \rm
In the proof above we used the assumption $\bw\in L^\infty(Q)^d$. This assumption can  be replaced by a different (more natural) assumption by using alternative estimates for the trilinear form $(\rho \bw\cdot \nabla \bu, \bv)$ which depend on the dimension $d$, see \cite[Section 2.3]{Temam}. The assumption $\bw \in L^\infty(Q)$ can be replaced by $\bw \in V$ for $d=2$. For $d=3$ we additionally need $\bw \in L^4(I;H^1(\Omega)^3)$.
\end{remark}

\begin{corollary}
Using that $C(\overline{I};L^2(\Omega)^d)$ is dense in $L^2(I;L^2(\Omega)^d)$ one can now derive the following well-posedness result: for each  $\bbf \in L^2(I;L^2(\Omega)^d)$ there exists a unique $\bu \in U_0$ such that \eqref{Res1} and \eqref{Res2} hold.
\end{corollary}
\subsection{Well-posed  space-time variational formulation in $V$} \label{SectweakV}
We define $V^0:=\{\, \bv \in V~|~\bv(0)=0\,\}$, where the trace is well-defined due to \eqref{rrt}. As an easy consequence of the result obtained in Theorem~\ref{thm1} we obtain the following.
\begin{corollary} \label{cor1}
Let $a(\cdot,\cdot)$ be a continuous elliptic bilinear form on $\cV \times \cV$ that does not depend on $t$. For every $F \in X'$ there exists a unique $\bu \in V^0$ such that
\begin{equation} \label{Res1a}
 \la \rho \dot \bu, \bv \ra + \int_0^T a(\bu(t),\bv(t)) \, dt = F(\bv) \quad \text{for all}~~\bv \in X.
\end{equation}
Furthermore
\begin{equation} \label{Res2a}
  \|\bu\|_{V} \leq c \|F\|_{X'}
\end{equation}
holds, with a constant $c$ independent of $F$.
\end{corollary}
\begin{proof}
 Take $F \in X'$. Due to the density of $C(\bar I;L^2(\Omega)^d)$ in $X'$ we can take a sequence $\bbf_n \in C(\bar I;L^2(\Omega)^d)$, $n \in \Bbb N$, with $\lim_{n \to \infty} \bbf_n  =F $ in $X'$. Let $\bu_n \in U_0$ be the unique solution of \eqref{Res1}. As test function we take $\bv=\bu_n$ in \eqref{Res1}. Using partial integration, $\bu_n(0)=0$ and ellipticity of $a(\cdot,\cdot)$ we get $\gamma \|\bu_n\|_X^2 \leq \|\bbf_n\|_{X'} \|\bu_n\|_X$, with ellipticity constant $\gamma >0$, and thus $\|\bu_n\|_X \leq \gamma^{-1} \|\bbf_n\|_{X'}$. This implies that $(\bu_n)_{n \in \Bbb N}$ is a Cauchy sequence in $X$. Take $\bu \in X$ such that $\lim_{n \to \infty} \bu_n= \bu$ in $X$.  Note that
 \begin{equation} \label{uu}
  \la \rho \dot \bu_n, \bv\ra  = (\rho \dot \bu_n, \bv)_{L^2}= - \int_0^T a(\bu_n(t),\bv(t)) \, dt + \int_0^T (\bbf_n(t),\bv(t))_{L^2(\Omega)} \, dt~~\forall~\bv \in X.
 \end{equation}
Hence, $\|\rho \dot \bu_n\|_{X'} \leq c (\|\bu_n\|_X + \|\bbf_n\|_{X'})$. This implies that $(\rho \dot \bu_n)_{n \in \Bbb N}$ is a Cauchy sequence in $X'$. Therefore $(\bu_n)_{n \in \Bbb N}$ is a Cauchy sequence in $V$ and $\lim_{n \to \infty} \rho \bu_n = \rho \dot \bu$ in $X'$ holds. Thus we get $\lim_{n \to \infty} \bu_n = \bu $ in $V$. From this and the trace inequality \eqref{rrt} we get $\bu(0)=0$, hence $\bu \in V^0$. If in \eqref{uu} we take $n \to \infty$ it follows that $\bu$ satisfies \eqref{Res1a}. Uniqueness of  $\bu$ follows by taking $F=0$ and $\bv=\bu$ in \eqref{Res1a}, partial integration identity \eqref{partint} and elliptcity of $a(\cdot,\cdot)$. From the estimates above we get $\|\bu_n\|_X + \|\rho \dot \bu_n\|_{X'} \leq c \|\bbf_n\|_{X'}$. Taking $n \to \infty$ we obtain the result in \eqref{Res2a}.
\end{proof}
\ \\[1ex]
If the (diffusion) coefficient $\mu$ in \eqref{weak1} would be constant, i.e., $\mu_1=\mu_2$ the result in Corollary~\ref{cor1} yields a well-posed weak formulation. In view of our applications, however,  the case $\mu_1 \neq \mu_2$ is  highly relevant. Therefore, in the remainder of this section we present an analysis that can handle the latter case. In that analysis the result derived in Corollary~\ref{cor1} will play an important role.
\ \\

%We take $\Gamma =1$ in \eqref{def2} and use the notation $\|\bv\|_\cV^2:=\hat a(\bv,\bv)$. Now note:
%\begin{align*}
% \sup_{t \in I} |a(t;\bu,\bv)-\hat a(\bu,\bv)| & \leq\sup_{t \in I} \int_{\Omega} \frac{|\alpha(x,t)- \alpha_{\max}|}{\alpha_{\max}} \alpha_{\max} |D(\bu):D(\bv)| \, dx \\
%  & \leq \sup_{(x,t) \in Q} \frac{|\alpha(x,t)- \alpha_{\max}|}{\alpha_{\max}} \|\bu\|_\cV \|\bv\|_\cV \\
%  & = \frac{\alpha_{\max}- \alpha_{\min} }{\alpha_{\max}}  \|\bu\|_\cV \|\bv\|_\cV = M  \|\bu\|_\cV \|\bv\|_\cV
%\end{align*}
%with $M= \frac{\alpha_{\max}- \alpha_{\min} }{\alpha_{\max}} <1$. Hence the bilinear form $a(t;\bu,\bv)$ is weakly $t$-dependent.

%{\bf IV: Important changes from here, till the end of the section\\}
For $F \in X'$ we consider the following generalization of the problem in \eqref{Res1a}. Determine $\bu \in V^0$ such that
\begin{equation} \label{defproblem}
 b(\bu,\bv):= \la \rho \dot \bu,\bv\ra + \int_0^T a(t;\bu(t),\bv(t)) \,dt = F(\bv) \quad \text{for all}~~\bv \in  X.
\end{equation}
In the remainder of this section we assume that the (possibly) $t$-dependent bilinear form $a(t;\cdot,\cdot)$ has the following properties:
\begin{align}
  \exists \, \gamma>0:\quad a(t;\bv, \bv) & \geq \gamma |\bv |_{1,\Omega}^2  \quad \text{for all}~~\bv \in \cV, ~t \in I, \label{Cond1} \\
 \exists \, \Gamma >0:\quad a(t;\bu, \bv) & \leq \Gamma |\bu |_{1,\Omega} |\bv |_{1,\Omega}  \quad \text{for all}~~\bu, \bv \in \cV, ~t \in I. \label{Cond2}
\end{align}
In the remainder we prove well-posedness of the variational problem in \eqref{defproblem}.  For this we first use the framework of the BNB-conditions, cf.~\cite{ErnGuermond}, to prove  well-posedness under the additional assumption that the bilinear form $a(t;\cdot,\cdot)$ is symmetric.  We then extend the well-posedness result to $a(t;\cdot,\cdot)$ that may be nonsymmetric.

From $ |\int_0^T a(t;\bu(t),\bv(t)) \, dt| \leq \Gamma \int_0^T |\bu(t) |_{1,\Omega} |\bv(t) |_{1,\Omega}\, dt \leq \Gamma \|\bu\|_X \|\bv\|_X$ for all $\bu, \bv \in X$ it follows that
\[
  |b(\bu,\bv)| \leq \sqrt{2} \max \{\Gamma,1\} \|\bu\|_V \|\bv\|_X \quad \text{for all}~~\bu \in V, \bv \in X.
\]
Hence $b(\cdot,\cdot)$ is continuous on $V^0 \times X$.

\begin{lemma}\label{la:infsup}
The inf-sup inequality
 \begin{equation}\label{infsup}
  \inf_{0\neq \bu \in V^0}~\sup_{ 0\neq \bv \in \overset{\phantom{.}}{X}} \frac{b(\bu,\bv)}{\|\bu\|_V\|\bv\|_X} \geq c_s 
 \end{equation}
holds with  $c_s= \frac{\sqrt{2}\, \gamma}{2(1+\Gamma^2)}$.
\end{lemma}
\begin{proof}
Take $\bu\in V^0$.
From the uniform ellipticity of $a(t;\cdot,\cdot)$ and the partial integration result \eqref{partint}, combined with $\bu(0)=0$, we get 
\begin{equation}\label{eq:proofinfsup1}
 b(\bu,\bu)= \la \rho \dot{\bu},\bu\ra +  \int_0^T a(t;\bu,\bu) \ge \gamma \|\bu\|_X^2.
\end{equation}
This establishes the control of $\|\bu\|_X$. We also need control of $\|\rho \dot \bu\|_{X'}$ to bound the full norm $\|\bu\|_V$.  This is achieved  by using a duality argument between the Hilbert spaces $X$ and $X'$.
By Riesz' representation theorem, there is a unique $\bz\in X$ such that $\la \rho \dot \bu,\bv\ra = (\bz, \bv)_X$ for all $\bv\in X$, and $\|\bz\|_X = \|\rho \dot{\bu}\|_{X'}$ holds. Thus we obtain
\[
  \la \rho \dot{\bu},\bz\ra = (\bz,\bz)_X = \|\rho \dot{\bu}\|_{X'}^2.
\]
Therefore, using the uniform continuity of $a(t;\cdot,\cdot)$,  we get
\begin{equation}\label{eq:proofinfsup2}
 \begin{split}
  b(\bu,\bz) &= \la \rho \dot{\bu},\bz\ra +  \int_0^T a(t;\bu(t), \bz(t))\, dt  = \|\bz\|_X^2 + \int_0^T a(t;\bu(t),\bz(t))\, dt \\ & \ge \|\bz\|_X^2 - \frac12 {\Gamma^2}  \|\bu\|_X^2 - \frac12\|\bz\|_X^2 
    = \frac12\|\rho\dot{\bu}\|_{X'}^2 - \frac12 {\Gamma^2}  \|\bu\|_X^2.
\end{split}
\end{equation}
This establishes control of $\|\rho \dot{\bu}\|_{X'}$ at the expense of the $X$-norm, which is controlled in \eqref{eq:proofinfsup1}.
Therefore, we make the ansatz $\bv= \bz + \delta \bu\in X$ for some sufficiently large parameter $\delta \geq 1$. We have the estimate
\begin{equation}\label{aux20}
\|\bv\|_X\le \|\bz\|_X + \delta \|\bu\|_X \le \delta \|\rho \dot \bu\|_{X'} + \delta\|\bu\|_X\le \delta \sqrt{2} \|\bu\|_V.
\end{equation}
 From \eqref{eq:proofinfsup1} and \eqref{eq:proofinfsup2} we conclude
\[
  b(\bu,\bv) \ge \frac12\|\rho \dot{\bu}\|_{X'}^2 +(\delta  \gamma - \frac12 \Gamma^2)\|\bu\|_X ^2.
\]
Taking $\delta:=\frac{1}{2\gamma} (1+\Gamma^2) \geq 1$, we get
\[
   b(\bu,\bv) \ge \frac12\|\bu\|_V^2 \ge \frac{\sqrt{2}}{4} \delta^{-1}\|\bu\|_V\|\bv\|_X.
\]
This completes the proof.
\end{proof}
\ \\
\begin{lemma} \label{la:leminj}
Assume that for all $t \in I$ the bilinear form $a(t;\cdot,\cdot)$ is symmetric on $X$. If $b(\bu,\bv)=0$ holds for  all $\bu \in V^0$, then $\bv=0$.
\end{lemma}
\begin{proof} Take $\bv \in X$ such that 
\begin{equation} \label{ll}
 b(\bu,\bv)= \la \rho \dot \bu,\bv\ra + \int_0^T a(t;\bu(t),\bv(t)) \,dt = 0 \quad \text{for all}~\bu \in V^0.
\end{equation} 
From Corollary~\ref{cor1}
 with $F(\bw):= \int_0^T \Gamma (\bv(t),\bw(t))_{1,\Omega} \, dt$, $\bw \in X$, it follows that there exists a unique $\bz \in V^0$ such that 
 \begin{equation} \label{hh}
\la \rho \dot \bz, \bw \ra + \int_0^T \Gamma (\bz(t),\bw(t))_{1,\Omega} \, dt = \int_0^T \Gamma (\bv(t),\bw(t))_{1,\Omega} \, dt  \quad \text{for all}~~\bw \in X.
\end{equation}
We take $\bw=\bz$ in \eqref{hh}, and use \eqref{partint}, $\bz(0)=0$. We get \[
 \Gamma \|\bz\|_{X}^2 \leq  \int_0^T \Gamma (\bv(t),\bz(t))_{1,\Omega}\, dt \leq  \Gamma \int_0^T (\bv(t),\bv(t))^\frac12_{1,\Omega} (\bz(t),\bz(t))^\frac12_{1,\Omega} \, dt \leq  \Gamma \|\bv\|_{X} \|\bz\|_{X}.
\]
Hence, $\|\bz \|_{X} \leq   \|\bv\|_{X}$ holds. 
Using \eqref{ll} and taking $\bw=\bv$ in \eqref{hh} we obtain:
\begin{equation}
\begin{split}
 \Gamma\|\bv\|_{X}^2 & = \la \rho \dot \bz, \bv \ra + \int_0^T \Gamma (\bz(t),\bv(t))_{1,\Omega} \, dt \\ & = \int_0^T \Gamma(\bz(t),\bv(t))_{1,\Omega} - a(t; \bz(t), \bv(t))\, dt \label{zandv}.
\end{split} \end{equation}
We define 
\[
S:=\{t\in I~|~ \bv(t)\neq 0~\text{and}~\bz(t)\neq 0 \}.
\]
If $S$ has measure 0, then \eqref{zandv} shows that $\bv=0$. Thus it suffices to prove that $|S|>0$ leads to a contradiction. Assume that $|S|>0$ holds. We apply, for $t \in S$, the Cauchy-Schwarz inequality to the symmetric positive semi-definite bilinear form $\Gamma (\cdot,\cdot)_{1,\Omega}-a(t,\cdot,\cdot)$ and use the ellipticity property \eqref{Cond1}: 
\[ \begin{split}
 \Gamma\|\bv\|_{X}^2   & =  \int_S \Gamma (\bz(t),\bv(t))_{1,\Omega} - a(t; \bz(t), \bv(t))\, dt 
 \\ &\leq \int_S \big(\Gamma |\bz(t)|_{1,\Omega}^2 - a(t; \bz(t), \bz(t))\big)^\frac12 \big(\Gamma |\bv(t)|_{1,\Omega}^2 - a(t; \bv(t), \bv(t))\big)^\frac12 \, dt
 \\ & \leq \int_S (\Gamma- \gamma)|\bz(t)|_{1,\Omega} |\bv(t)|_{1,\Omega}\, dt <    \Gamma \|\bz\|_{X} \|\bv\|_{X} \leq \Gamma \|\bv\|_{X}^2,
\end{split}\]
which results in a contradiction.
 Hence $\bv=0$ must hold. 
\end{proof}
\medskip

As a direct consequence of the preceding two lemmas and the continuity of $b(\cdot,\cdot)$ on $V^0 \times X$ we obtain the following main well-posedness result.

\begin{theorem} \label{mainprop1} Assume that $a(t;\cdot,\cdot)$ satisfies \eqref{Cond1}-\eqref{Cond2} and is symmetric. 
For any $F\in X'$, the problem \eqref{defproblem} has a unique solution $\bu\in V^0$. This solution satisfies the a-priori estimate
\begin{equation} \label{b0prop}
\|\bu\|_V \le c_s^{-1} \|F\|_{X'}, \quad \text{with}~~c_s=\frac{\sqrt{2}\, \gamma}{2(1+\Gamma^2)}. 
\end{equation}
\end{theorem}
\ \\[1ex]
We can apply this result to the time dependent bilinear form used  in the weak formulation of our original problem, cf.~\eqref{rt}. Hence, we obtain the following result, which \emph{shows well-posedness of the problem \eqref{weak1} with $W$ replaced by the (possibly) smaller subspace $V$}.
\begin{corollary} \label{corolmain}   For $F \in X'$ there exists a unique $\bu \in V^0$ such that 
\[
 \la \rho \dot \bu, \bv \ra + (\mu D(\bu),D(\bv))_{L^2} = F(\bv) \quad \text{for all}~~ \bv \in X.
 \]
 Furthermore $\|\bu\|_V \le c \|F\|_{X'}$ holds with a constant $c$ independent of $F$.
\end{corollary}
\ \\[1ex]
We derive a generalization of Theorem~\ref{mainprop1} in which the condition that $a(t;\cdot,\cdot)$ is symmetric is not needed.

\begin{theorem}\label{mainthm1}
Assume that $a(t;\cdot,\cdot)$ satisfies \eqref{Cond1}-\eqref{Cond2}. 
For any $F\in X'$, the problem \eqref{defproblem} has a unique solution $\bu\in V^0$. This solution satisfies the a-priori estimate
\begin{equation} \label{b0}
\|\bu\|_V \le c_s^{-1} \|F\|_{X'}, \quad \text{with}~~c_s=\frac{\sqrt{2}\, \gamma}{2(1+\Gamma^2)}. 
\end{equation}
\end{theorem}
\begin{proof} Recall the Neumann series result, that if $A \in \mathcal{L}(X,X)$ for some Banach space $X$ and $\|A\|_{\mathcal{L}(X,X)}<1$, then $I+A$ is an isomorphism on $X$ and  $(I+A)^{-1}\in \mathcal{L}(X,X)$ (see \cite[\S 5.7]{Alt2016}). 
We introduce some notation. Define the anti-symmetric part of $a(t;\cdot,\cdot)$: \[
c(t;\bu,\bv) := \frac12 a(t;\bu,\bv) - \frac12 a(t;\bv,\bu), \quad  \bu,\bv \in X.
\]
 We split the problem into a problem that we have treated in Theorem \ref{mainprop1}: $B\bu = b(\bu,\cdot)-\int_0^T c(t;\bu,\cdot)\in X'$ and a anti-symmetric part $C\bu = \int_0^T c(t;\bu,\cdot) \in X'$, hence \eqref{defproblem} has the operator representation $(B+C)\bu=F$. For $k\in \mathbb{N}$ we set $C_k:=\frac{1}{k}C$. Take $N\in  \mathbb{N}$ sufficiently large such that  $\|C_N\|_{\mathcal{L}(X,X')}\leq \frac{\gamma}{2}$ holds.
We prove the following statement by induction: for $k\in \mathbb{N}$ the operator $B+kC_N\in \mathcal{L}(V^0,X')$ is an isomorphism and $\|(B+kC_N)^{-1}\|_{ \mathcal{L}(X',X) } \leq \frac{1}{\gamma}$ holds.\\
For $k=0$ we can apply Theorem \ref{mainprop1}, because the symmetric part of $a(t;\cdot,\cdot)$ also satisfies \eqref{Cond1}-\eqref{Cond2}. Hence $B \in \mathcal{L}(V^0,X')$ is an isomorphism. The estimate $\|B^{-1}\|_{\mathcal{L}(X',X)} \leq \frac{1}{\gamma}$ follows from  \eqref{eq:proofinfsup1}.  We now treat the induction step. Assume that for given $k$ the statement holds. This implies
%\[
%\gamma \|B^{-1} f \|_X^2 \leq \la f ,B^{-1} f \ra \leq \|f \|_{X'}\|B^{-1} f\|_X
%\]
%which show the second part \textbf{IV: remove one of the two}. 
%We now do the induction step. By our choice of $N$ and the induction hypothesis we have that 
\[
\|C_N(B+kC_N)^{-1} \|_{\mathcal{L}(X',X')} \leq \|C_N\|_{\mathcal{L}(X,X')} \|(B+kC_N)^{-1} )\|_{\mathcal{L}(X',X)}\leq \frac{1}{2}
\]
and thus by the Neumann series result we get that $I+C_N(B+kC_N)^{-1} \in \mathcal{L}(X',X')$ is an isomorphism on $X'$. Using this, the induction hypothesis and the relation
\[
 B+(k+1)C_N=\big(I+ C_N(B+kC_N)^{-1}\big) (B+k C_N)
\]
it follows that $B+(k+1)C_N \in \mathcal{L}(V^0,X')$ is an isomorphism.  Using the antisymmetry property of $C$, i.e., $\la C_N\bu,\bu\ra=0$ and the ellipticity of $B$, cf.~\eqref{eq:proofinfsup1}, we get for arbitrary $\bu \in V^0$: % Now observe that for each $\bu \in V^0$ we have 
%\[
%b(\bu,\bu)-\frac{k+1}{N}\int_0^T c(t,\bu,\bu) = b(\bu,\bu),
%\]
 \[
\gamma \|\bu \|_X^2 \leq \la B\bu ,\bu \ra = \la (B+(k+1)C_N)\bu ,\bu \ra \leq \|(B+(k+1)C_N)\bu \|_{X'}\|\bu\|_X,
\]
hence, $\|(B+(k+1)C_N)^{-1}\|_{\mathcal{L}(X',X)} \leq \frac{1}{\gamma}$, which completes the induction. Taking $k=N$ we obtain that $B+C \in \mathcal{L}(V^0,X')$ is an isomorphism.
From \eqref{infsup} and $b(\bu,\cdot)=F$ we get \[
c_s \|\bu\|_V \leq \sup_{0\neq \bv \in X} \frac{b(\bu,\bv)}{\|\bv\|_X} =\|F\|_{X'},
\]
which completes the proof.
\end{proof}

\section{Space-time variational formulation in a broken space} \label{sectbroken}
In view of the fact that we want to use a DG method in time, we will now study a time-discontinuous weak formulation. Let $N\in \nn$, let $0=t_0<\dots<t_N=T$ and let $I_n=(t_{n-1},t_n)$ for $n=1,\dots,N$. 
For $\bv \in X$ we define $\bv_n:=\bv|_{I_n} \in X_n:=L^2(I_n;\cV) \subset X$, $1 \leq n \leq N$. Furthermore 
\[
V_n := \{ \bv_n~|~ \bv\in V \}, \quad 1 \leq n \leq N, ~~V^b:=\bigoplus_{n=1}^N V_n  \subset X.
\]
We define jumps at $t_n$ in the usual way. For $\bu \in V^b$: \[
[\bu]^n := \bu(t_n+) - \bu(t_n-)=: \bu_+^n- \bu_-^n, \quad 0 \leq n \leq N-1, ~~ \bu_-^{0}:=0.
\]
Note that the superscript $n$ denotes an evaluation at $t=t_n$, whereas $\bv_n$ denotes the restriction of $\bv$ to $I_n$.
Note that
\begin{equation} \label{V0V}
 V^0=\{\, \bv \in V^b~|~[\bv]^n=0, \quad 0 \leq n \leq N-1\,\}.
\end{equation}
For $\bu_n \in V_n$ we define
\[
 \la\rho \dot \bu_n, \bv \ra_n:= \la \rho \dot \bu, \bv_n \ra \quad \text{for all} ~~\bv \in X.
\]
Hence $\rho \dot \bu_n \in X_n'$.
\begin{remark}\rm On $Q_n:=I_n \times \Omega$ we can define a  set of smooth functions analogous to \eqref{defD0} by
\begin{align} \label{defD00}
 \cD_{0}^n:= \{\, \sum_{i=1}^m g_i \bphi_i~|~ m \in \Bbb{N}, \, g_i \in C_0^\infty(I_n),\, \bphi_i \in C_0^1(\Omega)^d \cap \cV\, \} \subset C_0^1(Q_n)^d
\end{align}
which is dense in $X_n$. Thus we get
\[
   \la\rho \dot \bu_n, \bphi \ra_n = - \int_{I_n} (\rho \bu_n(t), \dot \bphi(t))_{L^2} \, dt \quad \text{for all} ~~\bphi \in \cD_{0}^n. 
\]
Hence $\rho \dot \bu_n$ is the same weak material derivative as in Section~\ref{sectweak}, with $I$ replaced by $I_n$. Thus we have analogous results, e.g. as in \eqref{fundid}. In particular, for $\bu_n \in C^1(\bar{Q}_n)^d \cap X$ we have
\begin{equation} \label{res5}
 \la \rho \dot \bu_n, \bv\ra = \int_{I_n} (\rho \dot \bu_n(t),\bv(t))_{L^2} \, dt \quad \text{for all}~~\bv \in X.
\end{equation}
We also have
\[
  \la \rho \dot \bu, \bv \ra = \sum_{n=1}^N \la \rho \dot \bu,\bv_n \ra= \sum_{n=1}^N \la \rho \dot \bu_n,\bv\ra_n \quad \text{for all}~~\bu \in V, ~\bv \in X.
\]
Using $X'=L^2(I;\cV')= \oplus_{n=1}^N L^2(I_n;\cV')$ we get
\begin{equation} \label{idf}
 \|\rho \dot \bu\|_{X'}^2 = \int_{I} \|\rho \dot \bu(t)\|_{\cV'}^2\,dt = \sum_{n=1}^N\int_{I_n} \|\rho \dot \bu(t)\|_{\cV'}^2\,dt = \sum_{n=1}^N\|\rho \dot \bu(t)\|_{X_n'}^2 \quad \text{for}~\bu \in V.
\end{equation}
\end{remark}
\ \\

A broken weak time derivative is defined in the canonical way:
\[
  \la \rho \dot \bu, \bv \ra_b:= \sum_{n=1}^N \la \rho \dot \bu_n, \bv\ra_n, \quad \bu \in V^b,~ \bv \in X. 
\]
Hence,
\begin{equation} \label{eqder}
 \la \rho \dot \bu, \bv \ra_b= \la \rho \dot \bu, \bv \ra \quad \text{for all}~~\bu\in V,~\bv \in X.
\end{equation}

For controlling the jumps at the interval end points we introduce the usual discontinuous Galerkin bilinear form
\begin{equation}
 d(\bu,\bz):= \sum_{n=0}^{N-1} ([\bu]^n, \bz^n)_{L^2}, \quad \bu \in V^b, \bz^n \in L^2(\Omega)^d,
\end{equation}
with $\bz=(\bz^0, \ldots, \bz^{N-1}) \in (L^2(\Omega)^d)^N$. 

As test space in the weak formulation below we use $Y:=X \times H^N= \oplus_{n=1}^N (X_n\times H)$, where $H:=\overline{\V}^{L^2}$.
We consider the following weak formulation: given $F \in X', G\in H'$ determine $\bu \in V^b$ such that
\begin{equation} \label{brokenproblem}
\begin{split}
  B(\bu,(\bv,\bz)) & = F(\bv)+ G(\bz)\quad \text{for all} ~~(\bv,\bz) \in Y, \\
  \text{with}~~B(\bu,(\bv,\bz)) & := \la \rho \dot \bu, \bv\ra_b + d(\bu,\bz)+ \int_0^T a(t;\bu(t),\bv(t))\, dt .
\end{split} \end{equation}
Note that with $b(\cdot,\cdot)$ as in \eqref{defproblem} we have
\begin{equation} \label{ppA}
 B(\bu,(\bv,\bz))= b(\bu,\bv) \quad \text{for all}~~\bu \in V, ~(\bv,\bz) \in Y. 
\end{equation}
In the next theorem we derive equivalence results between different variational formulations.
%{\bf IV: Some changes in this proof, because $\rho\bv_+ \not\in H^N$. See also minor change in the definition of $d(\cdot,\cdot)$.}
\begin{theorem} \label{thmequivalence} Let the assumptions as in Theorem~\ref{mainthm1} be satisfied. For $F \in X'$ let $\bu\in V^0$ be the unique solution of \eqref{defproblem}. Then $\bu$ is also the unique solution of each of the following variational problems:\\
1. The problem \eqref{brokenproblem} with $G=0$. \\
2. Determine  $\bu \in V^b$ such that
\begin{equation} \label{vari2}
\la \rho \dot \bu, \bv\ra_b + d(\bu,\rho \bv_+)+ \int_0^T a(t;\bu(t),\bv(t))\, dt = F(\bv) \quad \text{for all} ~\bv \in V^b ,
\end{equation}
with $\rho \bv_+:=(\rho(t_0)\bv_+^0,\ldots, \rho(t_{N-1}) \bv_+^{N-1})$.
\end{theorem}
\begin{proof}
Let $\bu\in V^0$ be the unique solution of \eqref{defproblem}. Then $d(\bu,\bz)=0$ for all $\bz \in H^N$ and, cf.~\eqref{eqder}, $\la \rho \dot \bu, \bv \ra_b= \la \rho \dot \bu, \bv \ra$. Hence, $\bu \in V^0\subset V^b$ solves \eqref{brokenproblem} with $G=0$. Let $\bu \in V^b$ be a solution of \eqref{brokenproblem} with $G=0$. Taking $\bv=0$ we get $d(\bu,\bz)=0$ for all $\bz \in H^N$. This implies $[\bu]^n=0,~ 0 \leq n \leq N-1,$ and thus, cf. \eqref{V0V}, $\bu \in V^0$. Take $\bz =0$ and using $\la \rho \dot \bu, \bv \ra_b= \la \rho \dot \bu, \bv \ra$ we conclude that $\bu$ solves  \eqref{defproblem}. Hence, the unique solution  $\bu\in V^0$ of \eqref{defproblem} is also the unique solution of  \eqref{brokenproblem} with $G=0$.
\\
Let $\bu\in V^0$ be the unique solution of \eqref{defproblem}, which is also the unique solution of  \eqref{brokenproblem} with $G=0$. Taking arbitrary $\bv \in V^b \subset X$ and $\bz\in H^N$ such that $d(\cdot, \bz)= d(\cdot, \rho \bv_+ )\in (H^N)'$ in \eqref{brokenproblem} it follows that $\bu$ is a solution of \eqref{vari2}. Let $\bu \in V^b$ be a solution of \eqref{vari2}. The space $\{\, (\bv, d(\cdot, \rho\bv_+) )~|~ \bv \in V^b\,\}$ is dense in $X \times (H^N)'$. Note that $\bv \to \la \rho \dot \bu, \bv\ra_b$, $\bv \to \int_0^T a(t;\bu(t),\bv(t))\, dt$ are continuous functionals on $X$ and $\bz \to d(\bu,\bz)$ is continuous on $H^N$. Using a density argument it follows that $\bu$ solves \eqref{brokenproblem} with $G=0$.  Hence, the unique solution  $\bu\in V^0$ of \eqref{defproblem} is also the unique solution of \eqref{vari2}. 
\end{proof}
\ \\[1ex]
The factor $\rho$ in the coupling term  $d(\cdot,\cdot)$ in  \eqref{vari2} is not essential. It is introduced to obtain a natural scaling, namely one that corresponds to the scaling with $\rho$ in the weak time derivative. Note that in \eqref{vari2} the initial condition $\bu(0)=0$ is treated in a weak sense (applies also to $\bu(0)=\bu^0 \neq 0$).
\\
\begin{remark} \rm 
In Theorem~\ref{thmequivalence} we (only) show that the problem \eqref{brokenproblem} with $G=0$ has a unique solution. For the variational problem \eqref{brokenproblem} a more general well-posedness result can be derived, namely that the bilinear form $B(\cdot,\cdot)$ defines a homeomorphism  $V^b \to Y'$, with norms
\begin{align*}
  \|\bu\|_{V^b}^2 & :=\|\bu\|_X^2 +\sum_{n=1}^N \|\rho \dot \bu_n\|_{X_n'}^2+\big(\sum_{n=0}^{N-1}\|[\bu]^n\|_{L^2}\big)^2, \\
 \|(\bv,\bz)\|_Y^2 & =\|\bv\|_X^2+\big( \max_{0\leq n \leq N-1}\|\bz^n\|_{L^2}\big)^2.
\end{align*}
Note that $(V^b,\|\cdot\|_{V^b})$ and $(Y,\|\cdot\|_Y)$ are Banach spaces.
Continuity of the bilinear form $B(\cdot,\cdot)$ on $V^b \times Y$ is easy to show. Furthermore, provided $a(t;\cdot,\cdot)$ satisfies \eqref{Cond1}-\eqref{Cond2}, it can be shown that the BNB infsup conditions are satisfied. We do not include a proof in this paper. 
Given these results one obtains that under the above  assumptions on $a(t;\cdot,\cdot)$, for any $F\in X$, $G \in (H^N)'$ the problem \eqref{brokenproblem} has a unique solution $\bu\in V^b$ and the estimate 
\[
\|\bu\|_{V^b} \le c (\|F\|_{X'}^2+\|G\|_{(H^N)'}^2)^\frac12, 
\]
holds with a constant $c$ depending only on $\gamma,\, \Gamma$ from \eqref{Cond1}-\eqref{Cond2}.
\end{remark}
\\
\begin{remark} \rm
From the results above it follows that if the assumptions as in Theorem~\ref{mainthm1} are satisfied, then the weak formulation \eqref{vari2} is a well-posed variational formulation of the original Stokes problem \eqref{NSA}. This variational formulation, in which the same trial and test space $V^b$ is used,  can be reformulated using a time stepping procedure. The unique solution $\bu \in V^b$ of \eqref{vari2} can be decomposed as $\bu=(\bu_1,\ldots,\bu_N)$, with $\bu_n \in V_n$,
and  the solution of \eqref{vari2} is also the unique solution of the problem: for $n=1,\ldots,N$, determine $\bu_n \in V_n$ such that
\begin{equation} \label{timestepA} \begin{split}
  & \la \rho \dot \bu_n,\bv_n \ra_n +(\rho(t_{n-1}) \bu_n(t_{n-1}),\bv_+^{n-1} )_{L^2} + \int_{I_n} a(t;\bu_n(t),\bv_n(t))\, dt  \\ & = (\rho(t_{n-1})\bu_{n-1}(t_{n-1}), \bv_+^{n-1})_{L^2} + F(\bv_n) \quad \text{for all}~~ \bv_n \in V_n.
\end{split} \end{equation}
This is the usual form of a discontinuous Galerkin method for parabolic PDEs, cf. \cite{Thomee}. If $\bu_n$ has sufficient smoothness, e.g. $\bu_n \in C^1(\bar Q_n)\cap V_n$, the weak material derivative reduces to the usual strong one: $\la \rho \dot \bu_n,\bv_n \ra_n= \int_{I_n} (\rho \dot \bu_n(t),\bv_n(t))_{L^2} \, dt$. 
This formulation is a reasonable starting point for a Galerkin finite element discretization in which the space $V^b$ is  replaced by a (space-time) finite element subspace. This, however, requires exactly  divergence free finite element functions. Recently, such divergence free finite element methods have been further developed  using techniques from finite element exterior calculus, e.g. \cite{Falk}. Most finite element methods, however,  treat the divergence constraint by means of a pressure Lagrange multiplier, see  \cite{John}. Therefore, in Section~\ref{sectpressure} we introduce a variant of the weak formulation \eqref{vari2} that involves the pressure Lagrange multiplier to satisfy the divergence free constraint. 
\end{remark}
\ \\

%\section{Well-posed weak formulation with a pressure Lagrange multiplier} \label{sectpressure}
\section{Existence of a pressure Lagrange multiplier in $L^2(Q)$} \label{sectpressure}
In this section we reconsider the problem \eqref{defproblem}, for which a well-posedness result is given in Theorem~\ref{mainthm1}.  In the variational problem \eqref{defproblem}, both in the solution space $V^0$ and test space $X$ we restrict to functions $\bv$ which satisfy $\Div \bv =0$ on $\Omega$. In this section we derive a formulation in which we eliminate this condition from the trial and test space and instead introduce the pressure Lagrange multiplier for satisfying the divergence free constraint. For this one typically needs additional regularity properties of the solution $\bu $ of \eqref{defproblem}, cf. Section 6.2.1 in \cite{ErnGuermond}. The regularity property that we require in Theorem~\ref{thmwithp} below will be discussed in Remark~\ref{remregularity}. We use an analysis along the same lines as given for  a time dependent Stokes problem with constant coefficients (density and viscosity) in \cite{ErnGuermond}. 

We first introduce a space-time variant of de Rham's  theorem. Let $\nabla: \, L_0^2(\Omega) \to H^{-1}(\Omega)^d$ be the weak gradient. A standard application of de Rham's theorem, e.g., Corollary 2.4. in \cite{Reusken}, yields:
\begin{equation} \label{DeRham}
 \nabla : L^2_0(\Omega) \rightarrow \V^0 := \{f \in H^{-1}(\Omega)^d: f|_\V=0\} \quad \text{is an isomorphism},
\end{equation}
where $ L^2_0(\Omega) =\{\,p\in L^2(\Omega)~|~ \int_\Omega p=0\, \}$.
We define $\nabla_\otimes = {\rm id} \otimes \nabla:\, L^2(I;L_0^2(\Omega))=  L^2(I)\otimes L_0^2(\Omega) \rightarrow L^2(I;H^{-1}(\Omega)^d) = L^2(I) \otimes H^{-1}(\Omega)^d$ in the usual way, i.e., for $g \in  L^2(I;L_0^2(\Omega))$, $g(t)=\sum_{i=0}^\infty \alpha_i(t) \phi_i$ with $\alpha_i \in L^2(I)$, $\phi_i \in L_0^2(\Omega)$ we define
$(\nabla_\otimes g)(t):= \sum_{i=1}^\infty\alpha_i(t) \nabla \phi_i \in H^{-1}(\Omega)^d$. From \eqref{DeRham} it follows that
\begin{equation} \label{DeRham1}
 \nabla_\otimes : L^2(I;L^2_0(\Omega)) \rightarrow L^2(I;\V^0) \quad \text{is an isomorphism}.
\end{equation}
Furthermore, for $ g \in  L^2(I;L^2_0(\Omega)), \, \bv \in L^2(I;H_0^1(\Omega)^d)$ we have
\begin{equation}\label{partint2}
 \la \nabla_\otimes g, \bv\ra = \int_0^T \la \nabla_\otimes g(t), \bv(t)\ra_{H^{-1}(\Omega)} \,dt= - \int_0^T (g(t),\Div \bv(t))_{L^2(\Omega)} \, dt.
\end{equation}
We introduce notation for spaces. Recall $U=\{\, \bv \in X~|~\ddt{\bv} \in L^2(I;L^2(\Omega)^d)\, \}$, cf.~\eqref{defU}. We define
\[ \tilde U:=\{\, \bv \in L^2(I;H_0^1(\Omega)^d)~|~\ddt{\bv} \in L^2(I;L^2(\Omega)^d)\, \}. 
                                            \]
                                            Hence, $U=\{\, \bv \in \tilde U~|~\Div \bv(t)=0\quad \text{a.e. for}~t \in I\,\}$.
Clearly, opposite to $U$ and $V=\overline{U}^{\|\cdot\|_W}$, the space $\tilde U$ does not involve the divergence free constraint. Below we use this space as trial space and $L^2(I;H_0^1(\Omega)^d)$ (instead of $X$) as test space for the velocity. In order to do this we assume that the bilinear from $a(t;\cdot,\cdot)$ is not only defined on $\cV\times \cV$ but on $\cV\times H^1_0(\Omega)^d$ and satisfies 
 \begin{align} 
  a(t;\bv, \tilde\bv) & \leq \tilde \Gamma |\bv |_{1,\Omega} |\tilde \bv |_{1,\Omega}  \quad \text{for all}~~(\bv, \tilde \bv) \in \cV \times H^1_0(\Omega)^d, ~t \in I\label{condBnew}
 \end{align}
 for a positive constant  $\tilde \Gamma$, independent of $\bv,\tilde \bv$.
\begin{theorem} \label{thmwithp}
Let the assumptions of Theorem~\ref{mainthm1}  hold and assume that, for given $F \in L^2(I;H^{-1}(\Omega)^d) \subset X'$, the unique solution $\bu$ of \eqref{defproblem} has smoothness $\ddt{\bu}
 \in L^2(I;L^2(\Omega)^d)$, i.e., $\bu \in U$. Assume that the bilinear form $a(t;\cdot,\cdot)$ is defined on $\cV\times H^1_0(\Omega)^d$ and satisfies \eqref{condBnew}. Consider the following problem: determine $\bu \in \tilde U$, $p \in L^2(I;L_0^2(\Omega))$ such that
\begin{align}
 (\rho \dot \bu, \bv)_{L^2} +\int_0^T a(t;\bu(t),\bv(t))\, d t - \int_0^T (p(t),\Div \bv(t))_{L^2(\Omega)} \, dt & =F(\bv), \label{E1}\\
 \int_0^T (q(t),\Div \bu(t))_{L^2(\Omega)} \, dt &=0, \label{E2}
\end{align}
for all $\bv \in  L^2(I;H_0^1(\Omega)^d)$, $q \in L^2(I;L^2_0(\Omega))$. This problem has a unique solution $(\bu,p)$ and $\bu$ equals the unique solution of \eqref{defproblem}.
\end{theorem}
\begin{proof}
Let $\bu$ be the unique solution of  \eqref{defproblem}, which by assumption has smoothness $\ddt{\bu} \in L^2(I;L^2(\Omega)^d)$. Hence, $\la \rho \dot \bu,\bv \ra= (\rho \dot \bu, \bv)_{L^2}$ for all $\bv \in X$ and \eqref{E1} holds for all $\bv \in X$. Define
\[
 l(\bv):= F(\bv)-  (\rho \dot \bu, \bv)_{L^2}-\int_0^T a(t;\bu(t),\bv(t))\, d t, \quad \bv \in L^2(I;H_0^1(\Omega)^d).
\]
Then $l \in L^2(I;\V^0)$. From \eqref{DeRham1} it follows that there exists a unique $p \in L^2(I;L_0^2(\Omega))$ such that
\[
  \la \nabla_\otimes p,\bv\ra = l(\bv) \quad \text{for all}~ \bv \in L^2(I;H_0^1(\Omega)^d).
\]
Combining this with \eqref{partint2} we conclude that $(\bu,p)$ satisfies \eqref{E1} for all $\bv \in L^2(I;H_0^1(\Omega)^d)$. Furthermore, $\bu$ trivially satisfies \eqref{E2}, due to $\Div \bu(t)=0$. Hence,  the unique solution $\bu$ of  \eqref{defproblem} and the corresponding unique $p \in L^2(I;L_0^2(\Omega))$ solve \eqref{E1}-\eqref{E2} for all $\bv \in  L^2(I;H_0^1(\Omega)^d)$, $q \in L^2(I;L^2_0(\Omega))$.

We now consider the other direction. Let $(\bu,p)\in \tilde U \times L^2(I;L_0^2(\Omega))$ solve \eqref{E1}-\eqref{E2} for all $\bv \in  L^2(I;H_0^1(\Omega)^d)$, $q \in L^2(I;L^2_0(\Omega))$. From \eqref{E2} it then follows that $\Div \bu(t)=0$ a.e. for $t \in I$ and a.e. on $\Omega$. Hence, $\bu \in U$ holds. Taking $\bv \in X$ in \eqref{E1} it follows that $\bu$ must be equal to the unique solution of \eqref{defproblem}.
\end{proof}
\ \\[1ex]
%{\bf AR: Igor check whether this statement is correct. IV: Yes, it is.}
Since the unique solution of \eqref{defproblem} is also the unique solution of  \eqref{vari2} one can derive  the following \emph{time-discontinuous} variant of the space-time saddle point problem \eqref{E1}-\eqref{E2}. Define $\tilde U_n:=\tilde U_{|I_n}$, $\tilde U^b:= \oplus_{n=1}^N \tilde U_n$. The unique solution of \eqref{E1}-\eqref{E2} is also the unique solution of the following problem: determine $(\bu, p) \in \tilde U^b \times L^2(I;L_0^2(\Omega))$ such that
\begin{align}
 & ( \rho \dot \bu, \bv)_{L^2} + d(\bu,\rho \bv_+)+ \int_0^T a(t;\bu(t),\bv(t))\, dt - \int_0^T (p(t),\Div \bv(t))_{L^2(\Omega)} \, dt \nonumber \\ &   =F(\bv) \quad \text{for all}~\bv \in \tilde U^b, \label{vari2a}\\
 & \int_0^T (q(t),\Div \bu(t))_{L^2(\Omega)} \, dt =0 \quad \text{for all}~q \in L^2(I;L_0^2(\Omega)).\label{vari2b}
\end{align}
This allows a time stepping procedure, similar to \eqref{timestepA}. The formulation \eqref{vari2a}-\eqref{vari2b}, which allows a time-stepping procedure and treats the divergence free constraint by means of the pressure Lagrange multiplier, is a natural starting point for a \emph{Galerkin space-time finite element discretization}, which will be treated in the next section.
\\

\begin{remark} \label{remregularity}  \rm 
We briefly comment on  the regularity assumption  $\ddt{\bu} \in L^2(I;L^2(\Omega)^d)$  for the solution $\bu$ of \eqref{defproblem}, which is used in Theorem~\ref{thmwithp}.  One can derive (reasonable)  regularity conditions on the right hand-side functional $F \in X'$ and on the  given flow field $\bw$ that are sufficient  for the solution $\bu$ of \eqref{defproblem} to have the required smoothness $\ddt{\bu} \in L^2(I;L^2(\Omega)^d)$. For the derivation of such conditions one might consider to substitute $\bv =\dot \bu$ (or a smooth approximation of it) in \eqref{defproblem} and then use properties of $a(t;\cdot,\cdot)$ and smoothness assumptions on $F$ to derive a suitable bound for $\la \rho \dot \bu, \dot \bu\ra$ from which then $\dot \bu   \in L^2(I;L^2(\Omega)^d)$ can be concluded. This approach, however,  does not work, because we need test functions $\bv$ which are divergence free. The material derivative $\dot \bu$ of a divergence free function $\bu \in V$, however, is in general not 
divergence 
free. To circumvent this problem one can use a suitable Piola transformation or a Hanzawa transform as used  in \cite[Section 1.3]{Pruss2016}. We outline a result that can be derived using the Piola transformation. Details of the analysis are given in Appendix~\ref{AppendixReg}. 

Let $a(t;\cdot,\cdot)$ be as in Theorem \ref{mainthm1}, hence it satisfies \eqref{Cond1}-\eqref{Cond2}.
We furthermore assume  that this bilinear form is defined on $\cV\times H^1_0(\Omega)^d$, satisfies \eqref{condBnew} and 
 \begin{align} \label{condA}
 \int_0^T a(t;\bv(t),\dot{\bv}(t)) & \geq - M \|\bv \|_{X}^2 \quad \text{for all}~~\bv \in U_0\cap H^2(Q)^d,
 \end{align}
 for a positive constant  $M$, independent of $\bv$ (recall that $U_0=\{\, \bu \in U~|~\bu(0)=0\,\}$).  Then the unique solution $\bu\in V^0$ from Theorem \ref{mainthm1} has the (desired) smoothness property $\ddt{\bu} \in L^2(I;L^2(\Omega)^d)$, if $F'\in L^2(I;H)'$. This result can be derived (using a Piola transformation) as follows.

 We take a flow field with $\bw|_{\partial \Omega} = 0$.
We assume that  $\bw\in C^1(\bar I; C^2(\bar \Omega)^d)$ and consider  the corresponding Lagrange flow $\Phi: \Omega \times I \to \Omega $ as defined in \eqref{flowPhi},  
%\[
%\begin{cases}
%\ddt{  \Phi}(y,t) = \bw(\Phi(y,t),t), \quad t \in I,
%%\\
%\Phi(y,0)=y,
%\end{cases}
%\]
%where we assume that  $\bw\in C^1(\bar I, C^2(\Omega)^d)$ such that  this system has a unique solution 
which has smoothness $\Phi \in C^2(\bar Q)^d$ (see Remark \ref{rmSurfaceForce}).
The Lagrangian flow $\Phi$ defines a Piola transform $P_\cF$ of (time dependent) vector fields on $\Omega$. Both $P_\cF$ and its inverse $P_\cF^{-1}$ map (by construction of the Piola transform) divergence free functions to divergence free functions. Since $P_\cF$ is based on $\Phi$, it maps the  divergence free velocity field $\bu$ in Eulerian coordinates to a divergence free velocity field $P_\cF \bu$ in material coordinates. This allows us to define the following variant of the material time derivative $\bu' := P_\cF^{-1} \ddt{P_\cF \bu}$ which has the property that $\bu'$ is divergence free. One can verify that 
\begin{equation}\label{Pbounds}
\|\bu' - \dot \bu \|_X \leq C \| \bu \|_X, \text{ and }\|\bu' - \dot \bu \|_{L^2} \leq C \| \bu \|_{L^2}, 
\end{equation}
 for some $C$ which is independent of $\bu$. Formally using $\bu'$ as a test function in \eqref{defproblem} we obtain 
\begin{eqnarray*}
F(\bu')& =& \langle \rho \dot \bu , \bu' \rangle + \int_0^T a(\bu(t),\bu'(t))dt
\\& =& \| \sqrt{ \rho}  \dot \bu\|_{L^2}^2 +\langle \rho \dot \bu , \bu' - \dot \bu \rangle + \int_0^T a(\bu(t),\dot \bu(t))dt
+ \int_0^T a(\bu(t),\bu'(t) - \dot \bu(t))dt.
\end{eqnarray*}
From this one obtains, using \eqref{condA},\eqref{condBnew} and \eqref{Pbounds}, the estimate
$ \|  \dot \bu\|_{L^2}^2\leq c \|F\|_{L^2(I;H)'}^2 +c\|  \bu\|_{X}^2 $, which yields the desired smoothness result for $\bu$.
% where $c$ depends only on $C, \Gamma, M, c_\Phi$ and $c_s$. 
%We also note that the term $\|  \bu\|_{X}^2 $ can be controlled by the ellipticity of $a$.
 In order to justify the formal use of $\bu'$ as a test function in \eqref{defproblem} one can  construct a suitable sequence of sufficiently smooth functions that converge to $\bu$ (cf. Appendix~\ref{AppendixReg}).

It is not difficult to show that the conditions in \eqref{condBnew} and \eqref{condA} are satisfied for the bilinear form $a(t;\bu,\bv)= \int_\Omega \mu D(\bu):D(\bv)\, dx$. The latter condition is verified using $\dot \mu =0$, $D(\dot \bv) = \dot{\overbrace{D(\bv)}} + \nabla \bw \nabla \bv +( \nabla \bw \nabla \bv)^T$ and (a variant of) the integration by parts identity \eqref{partint}.
\end{remark}

\section{An unfitted space-time finite element method} \label{sectUnfittedFEM} 
In this section we introduce a Galerkin discretization scheme for \eqref{NSA}-\eqref{coupl2A}. In this scheme we use a standard space-time finite element space for the velocity approximation and a space-time cut finite element space for approximation of the pressure. The latter space is the same as the one used for a parabolic problem with a moving discontinuity in \cite{LRSINUM2013}.  A similar cut finite element spaces is used for stationary Stokes interface problems in  \cite{HansboStokes}. We explain the method and then present results of numerical experiments with this method.
\subsection{Discretization Scheme}
We wish to determine both the velocity and the pressure in \eqref{NSA}-\eqref{coupl2A}. We use the weak formulation \eqref{vari2a}-\eqref{vari2b} to formulate a space-time finite element discretization. We therefore take a pair of finite element spaces $U_h^b \subset \tilde U^b$, $Q_h \subset L^2(I;L_0^2(\Omega))$. These spaces are derived from  standard space-time tensor finite element spaces. 
For this we  assume a family of shape regular simplicial triangulations $\{\T_h\}_{h >0}$ of the (polygonal) spatial domain $\Omega$. The tensor product  mesh on the space-time domain is then given by 
\[
\M_{h,N} = \{ I_n \times \bT ~ |~ n=1,\dots, N, \bT \in \T_h\}.
\]
Standard space-time finite element spaces are:
\begin{eqnarray*}
Q_{h} &=& \{p \in L^2(I;H^1(\Omega)\cap L^2_0(\Omega)) ~|~ p|_{I_n\times \bT} \in \PP_q(I_n;\PP_{r-1}(\bT))~~\forall~ I_n\times \bT \in M_{h,N} \},
\\
U_{h}^b &=& \{\bu \in \tilde U^b~|~ \bu|_{I_n\times \bT} \in  \PP_q(I_n;\PP_r(\bT)^d)~~\forall~ I_n\times \bT \in M_{h,N},\},
\end{eqnarray*}
with integers $q \geq 0$, $r \geq 2$.  In both finite element spaces we use the same polynomial degree $q$ with respect to time. On each time-slab $ I_n\times \Omega $ the finite element functions in both spaces are  continuous on the entire slab. Note that in the space variable we have the $\PP_{r-1}-\PP_r$ Hood-Taylor pair. 
Clearly, using these spaces we can not expect an accurate  approximation of the jump in pressure. The large approximation errors in the pressure will induce large spurious velocities.  This can be remedied by using a suitable cut finite element variant of the pressure finite element space. Such spaces are well-known (in particular for stationary interface problems) in the literature and closely related to the extended finite element method (XFEM), cf.~\cite{LRSINUM2013,Burman2015}. This leads to the following definition of an extension of $Q_h$:
\[
 Q^X_h:= \mathcal{R}_1  Q_h \oplus  \mathcal{R}_2  Q_h \subset L^2(I;L_0^2(\Omega)),
\]
where $ \mathcal{R}_i:q\mapsto q \chi_{Q_i}$ is the restriction operator to the subdomain $Q_i:=\{ (x,t) \in  \Omega\times I_n~|~x \in \Omega_i(t)\,\}$, $i=1,2$. A similar extension of the velocity space $U^b_h$ could be considered. This, however, yields additional difficulties, because the continuity of the velocity across the interface has to be enforced weakly by a Nitsche method as in e.g. \cite{LRSINUM2013,Burman2015}. We will not do this here and leave  this topic for future research (cf. Section~\ref{sectoutlook}).   

A Galerkin discretization of the variational formulation \eqref{vari2a}-\eqref{vari2b}  
leads to  the following problem: determine $\bu_h \in U_h^b$, $p_h\in  Q_h^X$  such that
\begin{align}
(\rho \dot \bu_h, \bv_h)_{L^2}+d(\bu_h,\rho (\bv_h)_+) +\int_0^T a(t;\bu_h(t),\bv_h(t))\, d t
\\ - \int_0^T (p_h(t),\Div \bv_h(t))_{L^2(\Omega)} \, dt 
& =F(\bv_h), \label{E1h}\\
 \int_0^T (q_h(t),\Div \bu_h(t))_{L^2(\Omega)} \, dt &=0, \label{E2h}
\end{align}
for all $\bv_h \in  U_h^b$, $q_h \in Q_h^X$. 
This global in time problem 
can be solved sequentially by solving for each $n=1,\dots,N$, cf. \eqref{timestepA}: determine $\bu_{n,h} \in U_{h}^b|_{I_n}$, $p_{h,n}\in  (Q_h^X)|_{I_n}$ such that
\begin{equation} \label{timestepAh}
 \begin{split}
  & ( \rho \dot \bu_{h,n},\bv_{h,n} )_{L^2(Q_n)} +(\rho(t_{n-1}) \bu_{h,n}(t_{n-1}),(\bv_{h,n})_+^{n-1} )_{L^2(\Omega)} + \int_{I_n} a(t;\bu_{h,n}(t),\bv_{h,n}(t))\, dt  
 \\ & - \int_0^T (p_{h,n}(t),\Div \bv_{h,n}(t))_{L^2(\Omega)} \, dt   = (\rho(t_{n-1})\bu_{h,n-1}(t_{n-1}), (\bv_{h,n})_+^{n-1})_{L^2} + F(\bv_{h,n})
\\ & \int_0^T (q_{n,h}(t),\Div \bu_{n,h}(t))_{L^2(\Omega)} \, dt =0, 
\end{split} 
\end{equation}
for all $\bv_{h,n} \in U^b_h|_{I_n}$, $q_{n,h} \in Q_h|_{I_n}$, where $\bu_{n,h}=\bu_h|_{I_n}, p_{h,n}=q_h|_{I_n}$. 

Due to the fact that the triangulation is not fitted to the interface, special space-time quadrature is needed on the prisms that are cut by the interface $\mathcal{S}$. Moreover, the geometry of these cut elements has to be (approximately) determined. One typically uses a piecewise polygonal approximation of $\mathcal{S}$ for which the cut elements and corresponding quadrature rules can then be determined efficiently. Such an approach for the space-time setting is treated in \cite{Lehrenfeld2015}. These methods are used in the numerical experiments below.
\subsection{Numerical experiments}
We consider a problem with a prescribed smooth moving interface.
We take the space-time domain $I\times \Omega=(0,1)\times (-1,1)\times (-1,1) \times (-\frac{3}{4},\frac{7}{4} )$. We take a sphere which moves linearly in time, characterized as the zero level of the level set function
\[
\phi = x^2+y^2+(z-t)^2 - 1/2.
\]
The density and viscosity coefficients $\rho$ and $\mu$ are taken as follows:
\[
\rho  = \begin{cases}
1 & \phi > 0\\
10 & \phi < 0
\end{cases}, \quad
\mu  = \begin{cases}
1 & \phi > 0\\
25 & \phi < 0
\end{cases}, 
\]
and for the surface tension coefficient we take the value $\tau=2$. 
 The pressure solution is chosen to be smooth in the subdomains $Q_i$ and has a jump across $\mathcal{S}$
\[
p = \begin{cases}
0 & \phi > 0\\
\frac{96}{5}\sin(2t)xy +2\sqrt{2} & \phi < 0
\end{cases}.
\]
The velocty solution $\bu$ is chosen to be smooth in the entire domain:
\[
\bu= \sin(2t)  \begin{pmatrix}
\frac{1}{5} \, {\left(x^{2} + 5 \, y^{2} - 10 \, t z + 5 \, z^{2}\right)}  y
\\
\frac{1}{5} \, {\left(10 \, t^{2} + 5 \, x^{2} + y^{2} - 10 \, t z + 5 \, z^{2} - 8\right)} x
\\
\frac{4}{5} \, {\left(t - z\right)}  x y
\end{pmatrix}.
\]
In this first experiment we use a smooth velocity field because in the unfitted space-time finite element method introduced above  for the velocity variable we restrict to the standard finite element spaces (no CutFEM). 
We drop the advection term in \eqref{vari2a}-\eqref{vari2b} and take the bilinear forms as in the original problem. This corresponds to taking 
\[
a(t,\bu,\bv) = (\mu(t) D(\bu) : D(\bv) )_{L^2(\Omega)} - (\rho(t) \bw(t)\cdot\nabla \bu, \bv)_{L^2(\Omega)}
\]
in \eqref{E1h}-\eqref{E2h}. The obtained differential equation does no longer depend explicitly on $\bw$. The resulting PDE is defined by the position of the interface, which is given by $\phi(x,y,z,t)$.
\begin{remark} \rm 
Note that this bilinear form is not elliptic, however, it does satisfy 
\[
a(t,\bu,\bu) \geq \gamma |\bu|_{1}^2 - k_0 \|\bu \|_{L^2}^2
\]
for some $\gamma,k_0>0$ which depend on $\mu,\|\bw\|_{L^\infty}$. The standard transformation $\bu(t) \mapsto \exp(-\lambda_0 t) \bu(t)$, cf. \cite[p. 397]{Wloka} can be used in order to apply Theorem \ref{mainthm1}. 
\end{remark}
\ \\

The right hand-side $\mathbf{g}$ is adjusted to the prescribed solution and the surface tension force. We divide the interval $I$ into $N$ segments of length $k=\frac{1}{N}$. For the discretization in space we construct a tetrahedral triangulation of $\Omega$. For this  the domain $\Omega$ is divided into cubes with side length $h:=\frac{1}{N_S}$ and each of the cubes is divided into six tetrahedra. We use the finite element spaces $U_h,Q_h,Q_h^X$ which were introduced in the previous section. For the implementation of the surface tension forces and the pressure space $Q_h^X$ one needs an approximation of the interface. For this the level set function is interpolated by a piecewise bilinear function is space-time and the zero level of this interpolation is used as approximation for the interface. Further details concerning the space-time quadrature are given in \cite{Lehrenfeld2015}. Clearly this interface approximation limits the accuracy to second order. Therefore, in the finite element spaces we take 
$q=1$ (linears in time) and $r=2$ (linears for pressure, quadratics for velocity). 

%It is easy to commute that on the interface $x^2+y^2+(z-t)^2 = \frac{1}{2}$ we have \[
%D(\bu)\begin{pmatrix}
%x\\y\\z-t
%\end{pmatrix} = \frac{96}{5} \sin(2t)xy \begin{pmatrix}
%x\\y\\z-t
%\end{pmatrix}
%\]
%which implies that the corresponding right hand side is smooth.

Let $\bu_h,p_h$ be the solution of \eqref{E1h}-\eqref{E2h} in the spaces $U_h^b$ and $Q_h$ and  $\bu_h^X,p_h^X$  the solution of \eqref{E1h}-\eqref{E2h} in the spaces $U_h^b$ and $Q_h^X$. We determine errors in the $L^2\otimes H^1$ and the $L^2\otimes L^2$ norm. In Table \ref{TableH1L2} we show the error $\|\bu - \bu_h\|_{L^2\otimes H^1}$.

\begin{table}[ht!]
\begin{tabular}{r|lllllll}
$N_S$\textbackslash $N$ &4&8&16&32&64&128&$EOC_S$\\
\hline
4&0.32167&0.25902&0.25045&0.24888&0.24852&0.24844&\\
8&0.21797&0.13582&0.12703&0.12597&0.12577&0.12573&0.98264\\
16&0.18799&0.09090&0.08163&0.08095&0.08088&0.08086&0.63669\\
32&0.17626&0.06511&0.05425&0.05450&0.05466&0.05469&0.56423\\
$EOC_T$&&1.43670&0.26328&-0.00674&-0.00421&-0.00068&
\end{tabular}
\caption{\label{TableH1L2}
Error $\|\bu - \bu_h\|_{L^2\otimes H^1}$ for finite element spaces $U_h^b$ and $Q_h$.
The estimated temporal (spatial) order of convergence $EOC_T$ ($EOC_S$) is computed using the last row (column).
}
\end{table}

As expected, we observe poor convergence with a rate that is much lower than second order. In Table \ref{TableH1L2X} we see the error $\|\bu - \bu_h^X\|_{L^2\otimes H^1}$.
 \begin{table}[ht!]
\begin{tabular}{r|lllllll}
$N_S$\textbackslash $N$ &4&8&16&32&64&128&$EOC_S$\\
\hline
4&0.29649&0.20900&0.19753&0.19552&0.19510&0.19507&\\
8&0.18572&0.06802&0.04699&0.04390&0.04332&0.04318&2.17546\\
16&0.17339&0.04718&0.01736&0.01154&0.01064&0.01047&2.04410\\
32&0.17604&0.04423&0.01326&0.00525&0.00306&0.00267&1.97412\\
$EOC_T$&&1.99289&1.73737&1.33761&0.77615&0.20161&
\end{tabular}
\caption{\label{TableH1L2X}
Error  $\|\bu - \bu_h^X\|_{L^2\otimes H^1}$ for finite element spaces $U_h^b$ and $Q_h^X$.
The estimated temporal (spatial) order of convergence $EOC_T$ ($EOC_S$) is computed using the last row (column).
}
\end{table}

The error is roughly of optimal order $\cO(k^2+h^2)$. Note that the spatial error dominates after a few temporal refinements. We see that in absolute values the error significantly improves if we use the extended finite element space $Q_h^X$ for the pressure. 
In Tables \ref{TablePL2L2} and  \ref{TablePL2L2X} we give $L^2\otimes L^2$ norms of the pressure errors. 
\begin{table}[ht!]
\begin{tabular}{r|lllllll}
$N_S$\textbackslash $N$ &4&8&16&32&64&128&$EOC_S$\\
\hline
4&2.30681&2.22456&2.20904&2.20538&2.20451&2.20431&\\
8&1.58330&1.57570&1.57848&1.57931&1.57950&1.57956&0.48080\\
16&1.18787&1.15924&1.17049&1.17305&1.17346&1.17348&0.42873\\
32&0.93573&0.83261&0.83562&0.84323&0.84483&0.84507&0.47365\\
$EOC_T$&&0.16845&-0.00521&-0.01308&-0.00273&-0.00041&
\end{tabular}
\caption{\label{TablePL2L2}
Error $\|p - p_h\|_{L^2\otimes L^2}$ for  finite element spaces $U_h^b$ and $Q_h$.
The estimated temporal (spatial) order of convergence $EOC_T$ ($EOC_S$) is computed using the last row (column).
}
\end{table}

\begin{table}[ht!]
\begin{tabular}{r|lllllll}
$N_S$\textbackslash $N$ &4&8&16&32&64&128&$EOC_S$\\
\hline
4&1.81460&1.11767&0.99726&0.96497&0.97083&0.98616&\\
8&0.60974&0.26134&0.17501&0.16322&0.16244&0.16315&2.59559\\
16&0.72751&0.27060&0.10682&0.04883&0.04331&0.04355&1.90560\\
32&1.79975&0.40515&0.17825&0.07309&0.02646&0.01895&1.20034\\
$EOC_T$&&2.15127&1.18458&1.28605&1.46568&0.48186&
\end{tabular}
\caption{\label{TablePL2L2X}
Error $\|p - p_h^X\|_{L^2\otimes L^2}$ for  finite element spaces $U_h^b$ and $Q_h^X$.
The estimated temporal (spatial) order of convergence $EOC_T$ ($EOC_S$) is computed using the last row (column).
}
\end{table}
 We observe that for the space $Q_h$ the error $\|p-p_h\|_{L^2}$ has a poor spatial order $\cO(h^{\frac{1}{2}})$. This is known from the stationary case, cf. \cite[Section 7.10]{Reusken}. This spatial error dominates and we therefore see no temporal convergence order. If we use the space $Q_h^X$, then we see a significant improvement (Table \ref{TablePL2L2X}), however we do not see an optimal convergence rate $\cO(k^2+h^2)$. It is unclear what the temporal convergence rate is. The observed spatial convergence rate is consistent with results from stationary simulations, e.g. \cite[Table 7.17]{Reusken}. The spatial convergence order $\cO(h^{1.5})$ is probably caused by a dominating error in the approximation of the  surface tension. In  \cite{Grande2015} it is shown that the surface tension approximation method that we use induces a discretization error of order $\cO(h^{1.5})$. Additionally, since the finite element pair $U_h^b$-$Q_h^X$
 is not (necessarily) LBB-stable, a stabilization term would be beneficial, see
  \cite{HansboStokes,Kirchhart2016}. Such methods are a topic of ongoing research.

We consider a second experiment with a non-smooth velocity. 
We take the same $Q_i$, $\cS$ and $\phi$ as in the previous example. 
The density and viscosity coefficients are taken as follows:
\[
\rho  = \begin{cases}
1 & \phi > 0\\
5 & \phi < 0
\end{cases}, \quad
\mu  = \begin{cases}
1 & \phi > 0\\
2 & \phi < 0
\end{cases}, 
\]
and for the surface tension coefficient we take the value $\tau=2$. 
 The pressure solution and the velocity solution is chosen to be smooth in the subdomains $Q_i$ 
\[
p = \begin{cases}
0 & \phi > 0\\
2\sqrt{2} & \phi < 0,
\end{cases}
\quad
\bu= \sin(2t)  \begin{pmatrix}
-y
\\
x
\\
0
\end{pmatrix}\cdot
\begin{cases}
\frac{1}{2} \, e^{-{\left(t - z\right)}^{2} - x^{2} - y^{2}} & \phi > 0\\
-\frac{1}{2} \, e^{-\frac{1}{2}} + e^{-{\left(t - z\right)}^{2} - x^{2} - y^{2}}& \phi < 0.
\end{cases}
\]
\begin{table}[ht!]
\begin{tabular}{r|lllllll}
$N_S$\textbackslash $N$ &4&8&16&32&64&128&$EOC_S$\\
\hline
4&0.13430&0.04973&0.04391&0.04261&0.04213&0.04199&\\
8&0.07766&0.02683&0.01856&0.01771&0.01757&0.01758&1.25641\\
16&0.08380&0.02728&0.01185&0.00828&0.00787&0.00784&1.16538\\
32&0.11604&0.03314&0.01404&0.00696&0.00492&0.00473&0.72833\\
$EOC_T$&&1.80811&1.23931&1.01213&0.50098&0.05590&
\end{tabular}
\caption{\label{ConvReview}
Error  $\|\bu - \bu_h^X\|_{L^2\otimes H^1}$ for finite element spaces $U_h^b$ and $Q_h^X$.
The estimated temporal (spatial) order of convergence $EOC_T$ ($EOC_S$) is computed using the last row (column).
}
\end{table}

In Table \ref{ConvReview} we see that the space-time convergence order for the velocity is initially between 1 and $1.5$ and eventually it degrades towards the asymptotic order $0.5$. Similar behaviour has been observed in the stationary case, see \cite{Kirchhart2016,Fries2010}.
As expected,  the method would benefit from the use of a CutFEM space for the velocity unknown, see \cite{LRSINUM2013,Burman2015}. This a topic  for future research.

\section{Summary and outlook} \label{sectoutlook}
We have studied a time dependent Stokes problem that is motivated by  a standard sharp interface model for the fluid dynamics of  two-phase flows. This  Stokes interface problem has discontinuous  density and viscosity coefficients and a pressure solution that is discontinuous across the evolving interface. We consider this strongly simplified two-phase Stokes equation to be a good model problem for the development and analysis of finite element discretization methods for two-phase flow problems. Well-posedness results for this Stokes interface problem are not known in the literature.  We introduce  (natural) space-time variational formulations in a Euclidean setting and derive well-posedness results for these formulations. Different variants are considered, namely one with suitable spaces of divergence free functions, a discrete in time version of it, and variants in which the divergence free constraint in the solution space is treated by a pressure Lagrange multiplier. Although techniques known from the 
literature are used, the approach applied in the analysis of well-posedness is significantly different from known analyses of well-posedness of time-dependent (Navier-)Stokes problems. The reason for this is explained in Remark~\ref{Remdifficulty}. The discrete-in-time variational formulation involving the pressure variable for the divergence free constraint is a very natural starting point for a space-time finite element discretization. Such a  method, based on a standard DG time-stepping scheme and a special space-time extended finite element space (XFEM) for the pressure, is explained and  results of numerical experiments with this method are presented.  

In forthcoming work the following topics could be addressed.
A modified  analysis of well-posedness  may be possible which needs  weaker regularity requirements on $\bw$. This then leads to a smaller gap between regularity of $\bw$ and the regularity of the solution $\bu$. This is especially challenging for the regularity of $\bu$ and $\bw$ which is required to solve the full problem involving the pressure unknown.
The finite element  method can (and should) be combined with further methods which are already used in a stationary setting. For example, a stabilization term can be introduced for the pressure unknown to improve the conditioning of the stiffness matrix. Furthermore,  a Nitsche-XFEM method can be developed to treat problems in which the velocity is nonsmooth across the interface (which is typically the case). Another topic which we consider to be highly interesting for future research is an error analysis of the finite element method.  
\\[2ex]

\appendix

\section{A regularity result}\label{AppendixReg}

In this section we address the regularity assumption  $\ddt{\bu} \in L^2(I;L^2(\Omega)^d)$  for the solution $\bu$ of \eqref{defproblem}, which is discussed in Remark \ref{remregularity}. We will show that assumption \eqref{condBnew}, \eqref{condA}, together with the regularity assumptions $F \in L^2(I;H)'$ and $\bw\in C^1(\bar I;C^2(\bar \Omega)) \cap X$ imply the required smoothness $\ddt{\bu} \in L^2(I;L^2(\Omega)^d)$. 
%The idea of the analysis is as follows. One might consider to substitute $\bv =\dot \bu$ (or a smooth approximation of it) in \eqref{defproblem} and then use properties of $a(\cdot,\cdot)$ and smoothness assumptions on $F$ to derive a suitable bound for $\la \rho \dot \bu, \dot \bu\ra$ from which then $\dot \bu   \in L^2(I;L^2(\Omega)^d)$ can be concluded. This approach, however,  does not work, because we need test functions $\bv$ which are divergence free. The material derivative $\dot \bu$ of a divergence free function $\bu \in V$, however, is in general not divergence free. To circumvent this problem we use a suitable Piola transformation, which transforms $\bu$ in physical coordinates to the corresponding function $\hat \bu$ in Lagrangian coordinates, this technique has been used in \cite{Saal} for a free boundary problem. This transformation conserves the divergence free property. We take the partial time derivative of this function $\hat \bu$ and then transform back, using Piola, to physical 
%coordinates. The resulting function is divergence free and we will show that the difference between this function and $\dot \bu$ is ``of lower order'' and can be controlled. We then use this function (instead of $\dot \bu$) as test function in \eqref{defproblem} and using this equation derive a bound for $\|\dot \bu\|_{L^2}$. 

Let $\Phi $ be the Lagrangian flow of $\bw$ as in \eqref{flowPhi} . We will denote $\Phi(\cdot,t)$ by $\Phi_t$. The corresponding inverse mapping is given by $\Phi_t^{-1}(x)=y$, $x \in \Omega$. Note that the mapping $(x,t) \to \Phi_t^{-1}(x)$ has 
%the backward Lagrangian, we again have $\Phi_t^{-1} \in 
smoothness $ C^2(\bar Q)$. The Lagrangian mapping $\Phi_t$ induces a bijective $C^2$ diffeomorphism 
\[
  \cF:\, \bar Q=\bar \Omega \times \bar I \rightarrow \bar Q,\quad \cF(y,t):=(\Phi_t(y),t)=(x,t).
\]
By construction we have for any differentiable function $g$ on $Q$ that $\dot g = \ddt{(g \circ \cF)}\circ \cF^{-1}$, which expresses that $\dot g$ is the material derivative corresponding to the flow field $\bw$. As outlined in Remark \ref{remregularity}, while transforming coordinates based on the mapping $\cF$ we want to conserve the divergence free property of a vector function. For this we recall the Piola  transformation. For a given vector field $\bz \in H^1(\Omega)^d$ and a given diffeomorphism $\Psi: \, \Omega \rightarrow \Omega$,  the Piola mapping $P_\Psi$ is given by:
\[
 (P_\Psi \bz)(y):= \frac{1}{\det J\Psi(x)}J\Psi(x)\bz(x), \quad x \in \Omega,~~y := \Psi(x),
\]
where $J\Psi$ denotes the Jacobian of $\Psi$. 
This mapping has the property
\[
\Div (P_\Psi \bz)(y)= \frac{1}{\det J\Psi(x) } \Div \bz(x), \quad x\in \Omega.
\]
We introduce an isomorphism $P_\cF: \, L^2(I;L^2(\Omega)^d) \rightarrow  L^2(I;L^2(\Omega)^d)$, which is the application, for each $t \in I$, of the Piola transformation with $\Psi= \Phi_t^{-1}$:
\[
  (P_\cF\bu)(y,t)= \frac{1}{\det J\Phi_t^{-1}(x)} J\Phi_t^{-1}(x)\bu(x,t)=: A(x,t)\bu(x,t) ,
\]
with $y=\Phi_t^{-1}(x)$, $A(x,t)= \frac{1}{\det J\Phi_t^{-1}(x)} J\Phi_t^{-1}(x)$.  More compactly, we can write $P_\cF \bu = (A\bu)\circ \cF$. Its inverse is given by $P_\cF^{-1}\bu =A^{-1}\bu \circ \cF^{-1}$. Note that if $\Div \bu(x,t)=0$ then $\Div (P_\cF\bu)(y,t)=0$. For $\bu\in U_0$, we define
\[
 \bu':= P_\cF^{-1}(\frac{\partial }{\partial t} P_\cF\bu)= A^{-1} \dot A \bu +\dot \bu=: R \bu +\dot \bu.
\]
An important point to note is that if $\bu$ is divergence free then $\bu'$ is also divergence free, i.e., $\bu' \in L^2(I;H)$ if $\bu \in U_0$. We also note that $\bw|_{\partial \Omega} = \bu|_{\partial \Omega}=0$ implies that $\bu'|_{\partial \Omega}=0$, if the latter is defined. Concerning  the regularity of $R=A^{-1} \dot A$, we note that  $A^{-1}\dot  A \in C^1(\bar Q)^{d\times d}$, which can be concluded from the following. Since $\cF,\cF^{-1}$ are $C^2$ diffeomorphisms, we obtain that $A^{-1}, A \in C^1(\bar Q)^{d\times d}$. Hence, it suffices to verify that 
% $A^{-1} \dot A \in C^1(\bar Q)^{d\times d}$, it is sufficient 
 $\dot{\overbrace{J\Phi_t^{-1}(x)}} \in C^1(\bar Q)^{d \times d}$. Note that $\dot{\overbrace{\Phi_t^{-1}(x)}} = 0$ and for any $i=1,\dots, d$
\[
0= \deriv{}{x_i} \dot{\overbrace{\Phi_t^{-1}(x)}}  =  \dot{\overbrace{\deriv{}{x_i} \Phi_t^{-1}(x)}} + \deriv{}{x_i} \bw(x,t) \cdot \nabla \Phi_t^{-1}(x).
\]
Using that $\bw \in C^1(\bar I; C^2(\bar \Omega)^d)$, we can conclude that $\dot{\overbrace{J\Phi_t^{-1}(x)}} \in C^1(\bar Q)^{d \times d}$.

Using these preliminaries we can derive the following theorem.\\

\begin{theorem} \label{thmregularity}
Let the assumptions of Theorem~\ref{mainthm1}  hold and let $\bu \in V^0$ be the unique solution of \eqref{defproblem}. 
Assume that the bilinear form $a(t;\cdot,\cdot)$ is defined on $\cV\times H^1_0(\Omega)^d$ and satisfies \eqref{condBnew}, \eqref{condA}. Furtermore, assume  that $F \in L^2(I;H)'$, $\bw\in C^1(\bar I;C^2(\bar \Omega)) \cap X$ .  Then $\bu$ has the smoothness property $\ddt{\bu} \in L^2(I;L^2(\Omega)^d)$.%, moreover the mapping $L^2(I;H)' \rightarrow U^0: F \mapsto \bu$ is continuous.
\end{theorem}
\begin{proof}
Define $H_{\{0\}}^1(I):= \{\, g\in H^1(I)~|~ g(0)=0\, \}$ and, for a parameter  $\delta >0$, the operator $\tilde L_\delta:
\, H_{\{0\}}^1(I) \rightarrow L^2(I)$ by $\tilde L_\delta:= \ddt{g} +\delta g$. One easily checks that $\tilde L_\delta$ is an isomorphism. Hence $\tilde L_{\delta,\otimes}:=\tilde L_\delta \otimes {\rm id}:\, H^1_{\{0\}}(I)\otimes \V \rightarrow L^2(I) \otimes \V = X$ is an isomorphism.
Furthermore, since $A\in C^1(\bar Q)^{d\times d}, \cF\in C^2(\bar Q)^{d+1}$, the isomorphism $P_\cF:X\rightarrow X$ is also an isomorphism on the space $H^1_{\{0\}}(I)\otimes \V = \{\, \bv \in H^1(I;H^1_0(\Omega)^d)~|~\div \bv =0,~~\bv(x,0)=0\, \}$. From these observations we conclude that for $L_\delta := P_\cF^{-1} \circ \tilde L_{\delta, \otimes} \circ P_\cF$ we have
\begin{equation} \label{isom}
L_\delta\,:\,H^1_{\{0\}}(I)\otimes \V  \rightarrow X, \quad L_\delta \bu=  \bu' + \delta \bu
\quad \text{is an isomorphism.}
\end{equation}
In the Hilbert space $U_0 =H^1_{\{0\}}(I)\otimes \V$ with the norm $\|\bu\|_U^2=\|\frac{\partial \bu}{\partial t}\|_{L^2}^2+\|\bu\|_X^2$, we take a
total (orthonormal) set denoted by $(\bv_k)_{k \geq 1}$, and define $Z_m:={\rm span}\{\bv_1, \ldots, \bv_m\}$. Since smooth functions are dense in $U_0$, we can  assume that $\bv_m \in H^2(Q)^d$. We consider the following problem: determine $\bu_m \in Z_m$  such that:
 \begin{equation} \label{AA}
  (\rho \dot \bu_m,L_{\delta} \bv )_{L^2} + \int_0^T a(t;\bu_m(t), L_{\delta}\bv(t) ) \, dt =F(L_{\delta}\bv) \quad \forall~\bv \in Z_m.
 \end{equation}
For the bilinear form on the left hand-side we introduce the notation $B(\bu,\bv):=(\rho \dot \bu,L_{\delta} \bv )_{L^2} + \int_0^T a(t;\bu(t), L_{\delta}\bv(t) ) \, dt $. Recall that $\bu'=R\bu+\dot \bu$ and note that
\[
 \|R\bu\|_{L^2(I;H_0^1(\Omega)^d)} \leq C \|\bu\|_X,~~\|R \bu\|_{L^2} \leq C\|\bu\|_{L^2}.
\]
Using this, the assumptions \eqref{condA}-\eqref{condBnew} and $(\rho \dot \bu,\bu)_{L^2} \geq 0$  for  $\bu \in Z_m$, we get:
\begin{align*}
 B(\bu,\bu) =& (\rho \dot \bu, \bu')_{L^2} +\delta (\rho \dot \bu,\bu)_{L^2} + \int_0^T a(t;\bu,\bu'+\delta \bu)\, dt \\
  \geq & \rho_{\min}\|\dot \bu\|_{L^2}^2 +\gamma \delta\|\bu\|_X^2 - \rho_{\max}\|\dot \bu\|_{L^2}\|R \bu\|_{L^2} 
 \\ & - \tilde \Gamma \|R \bu\|_{L^2(I;H_0^1(\Omega)^d)} \|\bu\|_X - M\|\bu\|_X^2\\
  \geq &\frac12\rho_{\min} \|\dot \bu\|_{L^2}^2 +\gamma \delta\|\bu\|_X^2 - c \|\bu\|_{X}^2 ,
 \\
  \geq &\frac14\rho_{\min} \|\frac{\partial \bu}{\partial t}\|_{L^2}^2 +\gamma \delta\|\bu\|_X^2 -  \rho_{min}\|\bw\|_{L^{\infty}}^2\|\bu\|_X^2 - c \|\bu\|_{X}^2 ,\quad \text{for all}~\bu \in Z_m,
\end{align*}
with a constant $c$ independent of $\bu$ and $\delta$. Now we take  $\delta >0$, sufficiently large, such that 
\[
 B(\bu,\bu) \geq \frac14\rho_{\min}\|\bu\|_U^2 \quad \text{for all}~\bu \in Z_m
\]
holds. Hence the problem \eqref{AA} has a unique solution $\bu_m \in Z_m$ and 
\[
 \|\bu_m\|_{U}^2 \leq 4 \rho_{\min}^{-1} \|F\|_{L^2(I;H)'}(\|\dot \bu\|_{L^2}+\|R\bu \|_{L^2} +\delta \|\bu\|_{L^2}) \leq c' \|F\|_{L^2(I;H)'}\|\bu_m\|_{U}
\]
holds for some $c'>0$, which depends on $\rho_{min},\delta,\|\bw \|_{L^{\infty}}$ and $C$. Hence, $(\bu_m)_{m \geq 1} \subset U_0 $ has a subsequence, also denoted by $(\bu_m)_{m \geq 1}$, which weakly converges to some $\bu\in U_0$:
\[
 \bu_m \rightharpoonup \bu \quad \text{in}~~X, \quad \ddt{ \bu_m} \rightharpoonup \ddt{\bu} \quad \text{in}~~L^2(Q).
\]
  We conclude that $\bu \in X$ has smoothness $ \ddt{\bu} \in L^2(I;L^2(\Omega)^d)$, and taking the weak limit in \eqref{AA} and using continuity it follows that $\bu$ satisfies
 \begin{equation} \label{AAA}
  (\rho \dot \bu,L_{\delta} \bv )_{L^2} + \int_0^T a(t;\bu(t), L_{\delta}\bv(t) ) \, dt =F(L_{\delta}\bv) \quad \forall~\bv \in U_0.
 \end{equation}
 Using the isomorphism property \eqref{isom} we finally conclude that $\bu$ satisfies
 \[
   (\rho \dot \bu, \bv )_{L^2} + \int_0^T a(t;\bu(t), \bv(t) ) \, dt = F(\bv) \quad \forall~\bv \in X,
 \]
 and thus coincides with the unique solution of \eqref{defproblem}.
\end{proof}

\bibliographystyle{siam}
\bibliography{LitWP}

\begin{thebibliography}{10}

\bibitem{Abels2007}
{\sc H.~Abels}, {\em On generalized solutions of two-phase flows for viscous
  incompressible fluids}, Interfaces and Free Boundaries,  (2007), pp.~31--65.

\bibitem{HoFAbels2016}
{\sc H.~Abels and H.~Garcke}, {\em Weak Solutions and Diffuse Interface Models
  for Incompressible Two-Phase Flows}, Springer International Publishing, Cham,
  2016, pp.~1--60.

\bibitem{Matthies}
{\sc N.~Ahmed, S.~Becher, and G.~Matthies}, {\em Higher-order discontinuous
  {Galerkin} time stepping and local projection stabilization techniques for
  the transient {Stokes} problem}, Computer Methods in Applied Mechanics and
  Engineering, 313 (2017), pp.~28--52.

\bibitem{Alt2016}
{\sc H.~W. Alt}, {\em Linear Functional Analysis}, Springer London, 2016.

\bibitem{Bansch2001}
{\sc E.~B\"{a}nsch}, {\em Finite element discretization of the
  {Navier}--{Stokes} equations with a free capillary surface}, Numerische
  Mathematik, 88 (2001), pp.~203--235.

\bibitem{Hansbo2009}
{\sc R.~Becker, E.~Burman, and P.~Hansbo}, {\em A {Nitsche} extended finite
  element method for incompressible elasticity with discontinuous modulus of
  elasticity}, Computer Methods in Applied Mechanics and Engineering, 198
  (2009), pp.~3352--3360.

\bibitem{BotheReusken}
{\sc D.~Bothe and A.~Reusken}, {\em Transport Processes at Fluidic Interfaces},
  Advances in Mathematical Fluid Mechanics, Birkh\"auser, 2017.

\bibitem{Burman2015}
{\sc E.~Burman, S.~Claus, P.~Hansbo, M.~G. Larson, and A.~Massing}, {\em
  {CutFEM}: Discretizing geometry and partial differential equations},
  International Journal for Numerical Methods in Engineering, 104 (2015),
  pp.~472--501.

\bibitem{Burman2012}
{\sc E.~Burman and P.~Hansbo}, {\em Fictitious domain finite element methods
  using cut elements: {II}. a stabilized {Nitsche} method}, Applied Numerical
  Mathematics, 62 (2012), pp.~328--341.

\bibitem{Hansbo2014}
\leavevmode\vrule height 2pt depth -1.6pt width 23pt, {\em Fictitious domain
  methods using cut elements: {III}. a stabilized {Nitsche} method for
  {Stokes}{\rq} problem}, ESAIM: Mathematical Modelling and Numerical Analysis,
  48 (2014), pp.~859--874.

\bibitem{Zou1998}
{\sc Z.~Chen and J.~Zou}, {\em Finite element methods and their convergence for
  elliptic and parabolic interface problems}, Numerische Mathematik, 79 (1998),
  pp.~175--202.

\bibitem{Crippa2009}
{\sc G.~Crippa}, {\em The flow associated to weakly differentiable vector
  fields}, Edizioni della Normale, Pisa, 2009.

\bibitem{Croce2010}
{\sc R.~Croce, M.~Griebel, and M.~A. Schweitzer}, {\em Numerical simulation of
  bubble and droplet deformation by a level set approach with surface tension
  in three dimensions}, International Journal for Numerical Methods in Fluids,
  62 (2010), pp.~963--993.

\bibitem{Solonnikov1995}
{\sc I.~V. Denisova and V.~A. Solonnikov}, {\em Classical solvability of the
  problem of the motion of two viscous incompressible fluids}, St. Petersburg
  Mathematical Journal, 7 (1996), pp.~755--786.

\bibitem{Solonnikov2012}
\leavevmode\vrule height 2pt depth -1.6pt width 23pt, {\em Global solvability
  of a problem governing the motion of two incompressible capillary fluids in a
  container}, Journal of Mathematical Sciences, 185 (2012), pp.~668--686.

\bibitem{DiPernaLions}
{\sc R.~J. DiPerna and P.-L. Lions}, {\em Ordinary differential equations,
  transport theory and {Sobolev} spaces}, Inventiones Mathematicae, 98 (1989),
  pp.~511--547.

\bibitem{ErnGuermond}
{\sc A.~Ern and J.~Guermond}, {\em Theory and Practice of Finite Elements},
  Springer New York, 2013.

\bibitem{Evans}
{\sc L.~Evans}, {\em Partial Differential Equations}, American Mathematical
  Society, 2010.

\bibitem{Falk}
{\sc R.~S. Falk and M.~Neilan}, {\em {Stokes} complexes and the construction of
  stable finite elements with pointwise mass conservation}, SIAM Journal on
  Numerical Analysis, 51 (2013), pp.~1308--1326.

\bibitem{Fries2010}
{\sc T.-P. Fries and T.~Belytschko}, {\em The extended/generalized finite
  element method: an overview of the method and its applications},
  International Journal for Numerical Methods in Engineering, 84 (2010),
  pp.~253--304.

\bibitem{Grande2015}
{\sc J.~Grande}, {\em Finite element discretization error analysis of a general
  interfacial stress functional}, {SIAM} Journal on Numerical Analysis, 53
  (2015), pp.~1236--1255.

\bibitem{Reusken}
{\sc S.~{Gro{\ss}} and A.~Reusken}, {\em Numerical Methods for Two-phase
  Incompressible Flows}, Springer Berlin Heidelberg, 2011.

\bibitem{SchwabStevenson1}
{\sc R.~Guberovic, C.~Schwab, and R.~Stevenson}, {\em Space-time variational
  saddle point formulations of {Stokes and Navier-Stokes} equations}, ESAIM:
  Mathematical Modelling and Numerical Analysis, 48 (2014), pp.~875--894.

\bibitem{Hansbo2002}
{\sc A.~Hansbo and P.~Hansbo}, {\em An unfitted finite element method, based on
  {Nitsche}{\rq}s method, for elliptic interface problems}, Computer Methods in
  Applied Mechanics and Engineering, 191 (2002), pp.~5537--5552.

\bibitem{HansboStokes}
{\sc P.~Hansbo, M.~G. Larson, and S.~Zahedi}, {\em A cut finite element method
  for a {Stokes} interface problem}, Applied Numerical Mathematics, 85 (2014),
  pp.~90--114.

\bibitem{John}
{\sc V.~John, A.~Linke, C.~Merdon, M.~Neilan, and L.~Rebholz}, {\em On the
  divergence constraint in mixed finite element methods for incompressible
  flows}, SIAM Review, 59 (2017), pp.~492--544.

\bibitem{Kirchhart2016}
{\sc M.~Kirchhart, S.~{Gro{\ss}}, and A.~Reusken}, {\em {Analysis of an {XFEM}
  Discretization for {Stokes} Interface Problems}}, SIAM Journal on Scientific
  Computing, 38 (2016), pp.~A1019--A1043.

\bibitem{Lehrenfeld2015}
{\sc C.~Lehrenfeld}, {\em The {Nitsche XFEM-DG} space-time method and its
  implementation in three space dimensions}, SIAM Journal on Scientific
  Computing, 37 (2015), pp.~A245--A270.

\bibitem{LRSINUM2013}
{\sc C.~Lehrenfeld and A.~Reusken}, {\em Analysis of a {Nitsche XFEM-DG}
  discretization for a class of two-phase mass transport problems}, SIAM
  Journal on Numerical Analysis, 51 (2013), pp.~958--983.

\bibitem{Olshanskii2018}
{\sc A.~Lozovskiy, M.~A. Olshanskii, and Y.~V. Vassilevski}, {\em {A
  quasi-Lagrangian finite element method for the Navier--Stokes equations in a
  time-dependent domain}}, Computer Methods in Applied Mechanics and
  Engineering, 333 (2018), pp.~55--73.

\bibitem{Nouri1995}
{\sc A.~Nouri and F.~Poupaud}, {\em An existence theorem for the multifluid
  {Navier-Stokes} problem}, Journal of Differential Equations, 122 (1995),
  pp.~71--88.

\bibitem{Nouri1997}
{\sc A.~Nouri, F.~Poupaud, and Y.~Demay}, {\em An existence theorem for the
  multi-fluid {Stokes} problem}, Quarterly of Applied Mathematics, 55 (1997),
  pp.~421--435.

\bibitem{Pruss2009}
{\sc J.~Pr\"{u}ss and G.~Simonett}, {\em On the two-phase {Navier-Stokes}
  equations with surface tension}, arXiv preprint arXiv:0908.3327,  (2009).

\bibitem{Pruss2011}
\leavevmode\vrule height 2pt depth -1.6pt width 23pt, {\em Analytic solutions
  for the two-phase {Navier-Stokes} equations with surface tension and
  gravity}, Springer Basel, 2011, pp.~507--540.

\bibitem{Pruss2016}
\leavevmode\vrule height 2pt depth -1.6pt width 23pt, {\em Moving Interfaces
  and Quasilinear Parabolic Evolution Equations}, Birkh{\"a}user, 2016.

\bibitem{Saal}
{\sc J.~Saal}, {\em Maximal regularity for the {Stokes} system on
  noncylindrical space-time domains}, Journal of the Mathematical Society of
  Japan, 58 (2006), pp.~617--641.

\bibitem{Takahashi2009}
{\sc J.~{San Mart\'{\i}n}, L.~Smaranda, and T.~Takahashi}, {\em {Convergence of
  a finite element/ALE method for the Stokes equations in a domain depending on
  time}}, Journal of Computational and Applied Mathematics, 230 (2009),
  pp.~521--545.

\bibitem{SchwabStevenson2}
{\sc C.~Schwab and R.~Stevenson}, {\em Fractional space-time variational
  formulations of ({Navier}--) {Stokes} equations}, SIAM Journal on
  Mathematical Analysis, 49 (2017), pp.~2442--2467.

\bibitem{Solonnikov2007}
{\sc V.~A. Solonnikov}, {\em On the problem of non-stationary motion of two
  viscous incompressible liquids}, Journal of Mathematical Sciences, 142
  (2007), pp.~1844--1866.

\bibitem{Steinbach2018}
{\sc O.~Steinbach and H.~Yang}, {\em Comparison of algebraic multigrid methods
  for an adaptive space--time finite-element discretization of the heat
  equation in {3D} and {4D}}, Numerical Linear Algebra with Applications,
  (2018).
\newblock http://dx.doi.org/10.1002/nla.2143.

\bibitem{Temam}
{\sc R.~Temam}, {\em {Navier}-{Stokes} Equations: Theory and Numerical
  Analysis}, North-Holland Publishing Company, 1977.

\bibitem{Thomee}
{\sc V.~Thom\'{e}e}, {\em {Galerkin} Finite Element Methods for Parabolic
  Problems}, Springer-Verlag New York, Inc., 2006.

\bibitem{Wloka}
{\sc J.~Wloka}, {\em {Partial Differential Equations}}, Cambridge University
  Press, 1987.

\end{thebibliography}

\end{document}